\documentclass[a4paper]{article}
\usepackage{hyperref}
\usepackage{amsmath}
\usepackage{amsfonts}
\usepackage{amsthm}
\usepackage{indentfirst}
\usepackage{geometry}

\newtheorem{theorem}{Theorem}[section]
\newtheorem{lemma}[theorem]{Lemma}

\newtheorem{prop}{Proposition}[section]

\newtheorem{remark}{Remark}[section]

\renewcommand{\theequation}{\thesection.\arabic{equation}}
\catcode`@=11
\@addtoreset{equation}{section}
\catcode`@=12

\begin{document}
\begin{center}
{\LARGE Reversed inequality of the Herbst-type and the related Euler-Lagrange system}
\end{center}

\vskip 5mm
\begin{center}

{\sc Tiantian Zhou  \quad Yutian Lei}

\end{center}

\vskip 5mm

\begin{abstract}
In 2008, Beckner (Proc. Amer. Math. Soc. 136(5), 1871-1885)
proved two inequalities of the Herbst type, which
are the critical forms of the Stein-Weiss inequality.
In 2018, Chen et al. (Tran. Amer. Math. Soc. 370(12), 8429-8450)
established the reversed Stein-Weiss inequality.
In this paper, we are concerned about its critical case and
give a reversed Herbst inequality. Namely,
	$$
	\left|\int_{\mathbb{R}^n}\int_{\mathbb{R}^n}|x-y|^{\alpha/q'-n}|y|^{\alpha/q'}g(x)h(y)dxdy\right|
	\geq C_{n,\alpha,p,q'}\|g\|_{L^{q'}(\mathbb{R}^n)}\|h\|_{L^p(\mathbb{R}^n)}
	$$
holds for any nonnegative functions $g \in L^{q'}(\mathbb{R}^n)$ and $h \in L^p(\mathbb{R}^n)$, where
	$n\geq 1$, $p, q' \in (0,1)$, $\alpha>n$ satisfying ${1}/{p}+{1}/{q'}-{2\alpha}/(q'n)=1$.
Such an inequality is not covered by the reversed Stein-Weiss inequality.
Meanwhile, we prove the existence of extremal functions of this inequality.
Finally, we study the Euler-Lagrange system satisfied by those extremal functions
	$$
		\left\{\begin{matrix}
		u(x)=\int_{\mathbb{R}^n}|x-y|^{\beta-n}v^{-p_2}(y)|y|^{\beta}dy,	\\
		v(x)=\int_{\mathbb{R}^n}|x-y|^{\beta-n}u^{-p_1}(y)|x|^{\beta}dy.
	\end{matrix}\right.
	$$
	We obtain necessary conditions for the existence of positive solutions,
and investigate their integrability and asymptotic behavior when $|x| \to 0$ and $|x| \to \infty$.
\end{abstract}

\noindent{\bf{Keywords}}: reversed Herbst inequality, extremal function,
Euler-Lagrange system
\par
\noindent{\bf{MSC2020}}: 26D15, 45G15, 45M20, 46N20, 49J21, 45M05

\renewcommand{\theequation}{\thesection.\arabic{equation}}
\catcode`@=11
\@addtoreset{equation}{section}
\catcode`@=12

	\renewcommand{\theequation}{\thesection.\arabic{equation}}
	\catcode`@=11
	\@addtoreset{equation}{section}
	\catcode`@=12

	\section{Introduction}
	
	The well-known Hardy-Littlewood-Sobolev inequality states that (cf. Theorem 1 in Chapter 5 of \cite{Stein})
    \begin{equation}\label{HLS}
		\left|\int_{\mathbb{R}^n}\int_{\mathbb{R}^n} \frac{f(x)g(y)dxdy}
	{|x-y|^\lambda}\right|
	\leq C_{n,\lambda,p}\|f\|_{L^p(\mathbb{R}^n)}\|g\|_{L^q(\mathbb{R}^n)}.
   \end{equation}
    for all $(f,g) \in L^p(\mathbb{R}^n) \times L^q(\mathbb{R}^n)$.
    Here $0<\lambda<n$, $\min\{p,q\}>1$,
\begin{equation}\label{sbv1}
1/p+1/q+\lambda/n=2,
\end{equation}
and $C_{n,\lambda,p}$ is the best constant. In 1983,
Lieb \cite{LEH} proved the existence of extremal functions of \eqref{HLS}.
Meanwhile, he obtained the best constant
	$$
	C_{n,\lambda,p}= \pi^{\frac{\lambda}{2}}\frac{\Gamma(\frac{n-\lambda}{2})}{\Gamma(\frac{2n-\lambda}{2})}
\{\frac{\Gamma(\frac{n}{2})}{\Gamma(n)}\}^{-\frac{n-\lambda}{n}}
	$$
	of \eqref{HLS} in the special case of $p=q=2n/(2n-\lambda)$, and in this case,	
\begin{equation}\label{exl}
f(x)=g(x)=c\left(\frac{a}{d+|x-\tilde{x}|^2}\right)^{\lambda/2},
\end{equation}
	with $a, d, c>0$ and $\tilde{x}\in \mathbb{R}^n$. These results show the sharp
upper bound of the Coulomb energy $\|(|x|^{-\lambda} \ast f)f\|_{L^1(\mathbb{R}^n)}$.
	
Let $\bar{g}$ and $\bar{f}$ be the rescaling transformation of $g$ and $f$ respectively.
Clearly, $(\bar{f},\bar{g}) \in L^{p}(\mathbb{R}^n) \times L^{q}(\mathbb{R}^n)$.
Now, \eqref{sbv1} ensures that $\bar{g}$ and $\bar{f}$ still satisfy \eqref{HLS}.
The condition that ensures the invariance of an inequality under the rescaling
transformation is called the {\it Sobolev-type condition}. Clearly, \eqref{sbv1} is the
Sobolev-type condition of \eqref{HLS}.

	In 1958, Stein and Weiss \cite{SW} proved the weighted Hardy-Littlewood-Sobolev inequality
	\begin{equation}\label{888}
		\left|\int_{\mathbb{R}^n}\int_{\mathbb{R}^n}
		\frac{f(x)g(y)}{|x|^{\alpha}|x-y|^{\lambda}|y|^{\beta}}dxdy\right|
		\leq C_{n,\alpha,\beta,p,q'}\|f\|_{L^{p}(\mathbb{R}^n)}\|g\|_{L^q(\mathbb{R}^n)}
	\end{equation}
	for all $(f,g)\in L^{p}(\mathbb{R}^n)\times L^q(\mathbb{R}^n)$,
	where $0<\lambda<n$, $\alpha+\beta \geq 0$, $\min\{p,q\}>1$, $1-1/p-\lambda/n<\alpha/n<1-1/p$,
and the Sobolev-type condition
\begin{equation}\label{sbv2}
1/q+1/p+(\alpha+\beta+\lambda)/n=2.
\end{equation}
Clearly, \eqref{sbv2} is the Sobolev-type condition of \eqref{888}.	
Lieb \cite{LEH} maximized the functional
	$$
	J(f,g)=\int_{\mathbb{R}^n}\int_{\mathbb{R}^n}\frac{f(x)g(y)}{|x|^{\alpha}|x-y|^{\lambda}|y|^{\beta}}dxdy
	$$
	under the constraint $\|f\|_{L^{p}(\mathbb{R}^n)}=\|g\|_{L^q(\mathbb{R}^n)}=1$
to show that the best constant $C_{n,\alpha,\beta,p,q'}$ in \eqref{888} is achievable.
The corresponding Euler-Lagrange system of $J(f,g)$ is
	\begin{equation}\label{889}
		\left \{
		\begin{array}{lll}
			u(x)=\displaystyle\int_{\mathbb{R}^n}\frac{v^{s}(y)dy}
			{|x|^{\alpha}|x-y|^{\lambda }|y|^{\beta}},   \\
			v(x)=\displaystyle\int_{\mathbb{R}^n}\frac{u^{r}(y)dy}
			{|x|^{\beta}|x-y|^{\lambda }|y|^{\alpha}},
		\end{array}
		\right.
	\end{equation}
	where
	$$
    \begin{cases}
    	\min\{r,s\}>0,\  0<\overline{\lambda}:=\alpha+\beta+\lambda<n,\  \min\{\alpha,\beta\}\geq 0,\\
    	\frac{\alpha}{n}<\frac{1}{r+1}<\frac{\lambda+\alpha}{n},\ \frac{\beta}{n}<\frac{1}{s+1}<\frac{\lambda+\beta}{n},\  \frac{1}{r+1}+\frac{1}{s+1}=\frac{\overline{\lambda}}{n}.
    \end{cases}
	$$

Although it is difficult to write explicit expressions of extremal functions of \eqref{888} as in \eqref{exl},
we can investigate the qualitative properties (such as the radial symmetry, the integrability,
and the asymptotic behavior when $|x| \to
0$ and $|x| \to \infty$) of solutions of \eqref{889} to show the geometric shapes of those extremal functions.
The conclusions of radial symmetry can be found in \cite{CLO-CPDE,CLO,JL,L}, and the
conclusions of integrability can be found in \cite{CJLL,CLO-DCDS,JL2,LL2}. Based on these results,
paper \cite{LLM} shows the asymptotic behavior when $|x| \to 0$ and $|x| \to \infty$.

	In 2008, Beckner proved the following inequalities (cf. Theorem 1 in \cite{BW})
which were used earlier by Herbst in the study of the Klein-Gordon equation for a Coulomb potential
(cf. Theorem 2.5 in \cite{H}).
	
\textit{Let $0<\alpha<n$, $1<q<\infty$ and $1/q+1/q'=1$. Then we have}
		\begin{equation}\label{BT2}
			\||x|^{-\alpha/q'}(|x|^{-(n-\alpha/q')}\ast g)\|_{L^{q'}(\mathbb{R}^n)}
			\leq C_{\alpha,q'}\|g\|_{L^{q'}(\mathbb{R}^n)},
		\end{equation}		
		\begin{equation}\label{BT1}
			\left \| |x|^{-(n-\alpha/q')} \ast (|x|^{-\alpha/q'}h) \right \|_{L^q(\mathbb{R}^n)} \leq C_{\alpha, q'}\left \| h \right \|_{L^q(\mathbb{R}^n)},
		\end{equation}
{\it for all $(g, h) \in L^{q'}(\mathbb{R}^n) \times L^q(\mathbb{R}^n)$.
		Here}
		$$
		C_{\alpha, q'}=\pi^{n/2}\left[\frac{\Gamma (\frac{\alpha}{2q'})\Gamma(\frac{n-\alpha}{2q'}) \Gamma (\frac{n}{2q})}{\Gamma (\frac{n}{2}-\frac{\alpha}{2q'})\Gamma (\frac{n}{2q}+\frac{\alpha}{2q'})\Gamma (\frac{n}{2q'})}\right].
		$$
We only consider \eqref{BT1} because \eqref{BT2} is equivalent to \eqref{BT1} by duality.	
	
	By the H\"{o}lder inequality and \eqref{BT1}, we have
	\begin{equation}\label{BI2}
			\int_{\mathbb{R}^n}\int_{\mathbb{R}^n}\frac{|g(x)||h(y)|}{|x-y|^{n-\alpha/q'}|y|^{\alpha/q'}}dydx
			\leq C \left \| h \right \|_{L^q(\mathbb{R}^n)}\left \| g \right \|_{L^{q'}(\mathbb{R}^n)}
			\end{equation}
for all $(g, h) \in L^{q'}(\mathbb{R}^n) \times L^q(\mathbb{R}^n)$.
Now, the Sobolev-type condition is
\begin{equation}\label{sbv5}
\frac{1}{q}+\frac{1}{q'}=1.
\end{equation}	
On the contrary, if setting $Th(x)=|x|^{-(n-\alpha/q')}\ast (|x|^{-\alpha/q'}h)$,
we can deduce \eqref{BT1} from \eqref{BI2} and the definition of the norm of operator $T$.

Here, \eqref{BI2} is called the {\it Herbst inequality}.

The form of \eqref{BI2} seems to be a special case of the Stein-Weiss
inequality \eqref{888}. In fact, \eqref{BI2} cannot be covered by \eqref{888} because $\alpha\beta=0$
may not hold when we note \eqref{sbv2} and $1-1/p-\lambda/n<\alpha/n<1-1/p$. In addition,
according to Remark (iii) in Page 369 of \cite{LEH}, we do not expect that
the best constant of \eqref{BI2} can be achieved.
	
\vskip 5mm
	
Next, we recall several corresponding reversed inequalities.

	In 2015,
	Dou and Zhu \cite{DZ} proved the reversed Hardy-Littlewood-Sobolev inequality
	(see also \cite{B2} and \cite{NN})
	\begin{equation}\label{vv6}
		\int_{\mathbb{R}^n}\int_{\mathbb{R}^n}\frac{|f(x)||g(y)|dxdy}{|x-y|^\lambda}
		\geq C_{n,\lambda,r} \|f\|_{L^r(\mathbb{R}^n)} \|g\|_{L^s(\mathbb{R}^n)}
	\end{equation}
for all $(f, g) \in L^r(\mathbb{R}^n) \times L^s(\mathbb{R}^n)$.
	Here $n \geq 1$, $r \in (n/(n-\lambda),1)$, $s \in (n/(n-\lambda),1)$
and $\lambda<0$ satisfy the Sobolev-type condition
\begin{equation}\label{sbv3}
\frac{1}{r}+\frac{1}{s}+\frac{\lambda}{n}=2,
\end{equation}
and $C_{n,\lambda,r}>0$ is the best constant. In addition, they also proved
that $C_{n,\lambda,r}$ is achievable. In the process of proof,
they mainly used the reversed H\"{o}lder inequality,
the reversed Young inequality and a new established Marcinkiewicz-type interpolation
involving exponents less than $1$ (could be negative).
The Euler-Lagrange system satisfied by the extremal functions is
\begin{equation}\label{j4}
		\left\{\begin{matrix}
\displaystyle			u(x)=\int_{\mathbb{R}^n}|x-y|^{\lambda}v^{-p_2}(y)dy,	\\
\displaystyle			v(x)=\int_{\mathbb{R}^n}|x-y|^{\lambda}u^{-p_1}(y)dy,
		\end{matrix}\right.
	\end{equation}
	where $n\geq 1$, $\min\{\lambda, p_1, p_2\}>0$, and \eqref{sbv3} becomes
	\begin{equation}\label{j5}
		\frac{1}{p_1-1}+\frac{1}{p_2-1}=\frac{\lambda}{n}.
	\end{equation}
	
	 When $u\equiv v$ and $p_1=p_2$, \eqref{j4} is reduced to
	\begin{equation}\label{j9}
		u(x)=\int_{\mathbb{R}^n}|x-y|^{\lambda}u^{-p_1}(y)dy,
	\end{equation}
	which is related to the study of the conformal geometry and the nonlinear elliptic PDEs.
	For $0<p_1 \leq (2n+\lambda)/\lambda$, Li \cite{L} proved that
	$p_1=(2n+\lambda)/\lambda$ and $u$ is classified as form
	\begin{equation}\label{j10}
		u(x)=a(b^2+|x-x_0|^2)^{\lambda/2},
	\end{equation}
	with $a,b>0$ and $x_0\in \mathbb{R}^n$. Afterwards,
Xu \cite{X} proved that \eqref{j9} has a positive solution
if and only if $p_1=(2n+\lambda)/\lambda$ when $\lambda>0$ and $p_1>0$,
and $u$ is given by \eqref{j10}.
    This shows a sharp lower bound of the Coulomb energy $\|(|x|^{\lambda} \ast f)f\|_{L^1(\mathbb{R}^n)}$
in the special case of $\lambda>0$. Moreover, Xu \cite{X} pointed out that\eqref{j9} has no positive solution
when $\lambda \in (-n,0)$ and $p_1>0$. In addition, the Coulomb-Sobolev inequality comes into play
in estimating the lower bound of the Coulomb energy with $\lambda<0$ and $p_1<0$
(cf. \cite{BFV,BGMMS,BGO,MMS}).

Clearly, \eqref{j5} is the critical condition of Sobolev-type. The results in \cite{DZ,NN}
show that \eqref{j5} is a necessary and sufficient condition of existence of positive solutions of \eqref{j4}.
When $p_1=p_2$, the results of radial symmetry can be found in \cite{DZ,L}.
     For \eqref{j4} with $p_1 \neq p_2$, Liu \cite{LZ} used the improved method of moving planes proposed by
     Dou, Guo and Zhu \cite{DGZ} to obtain the radial symmetry and monotonicity of the solutions.
Lemma 3 in \cite{NN} shows the asymptotic behavior of the solutions when $|x| \to \infty$.
For more related results about nonlinear equations with negative exponents,
we could refer to \cite{CX,GW,LCW,MW} and the references therein.

	In 2018, Chen et al. \cite{CLLT} proved the reversed Stein-Weiss inequality
	\begin{equation}\label{vv1}
		\int_{\mathbb{R}^n}\int_{\mathbb{R}^n}
		\frac{|f(x)||g(y)|dxdy}{|x|^\alpha|x-y|^\lambda|y|^\beta}
		\geq C_{n,\lambda,\alpha,\beta,r} \|f\|_{L^r(\mathbb{R}^n)} \|g\|_{L^s(\mathbb{R}^n)},
	\end{equation}
for all $(f,g) \in L^r(\mathbb{R}^n) \times L^s(\mathbb{R}^n)$.
	Here $n \geq 1$, $r \in (0,1)$, $s \in (0,1)$, $\lambda<0$, $\alpha \in
		(-n(1-r)/r,0]$ and $\beta \in (-n(1-s)/s,0]$ satisfy the Sobolev-type condition
\begin{equation}\label{sbv4}
\frac{1}{r}+\frac{1}{s}+\frac{\lambda+\alpha+\beta}{n}=2,
\end{equation}
and $C_{n,\lambda,\alpha,\beta,r}$ is the best constant.
Furthermore, they proved that $C_{n,\lambda,\alpha,\beta,r}$
is achievable. In addition, they also gave the asymptotic behavior of
positive solutions of the Euler-Lagrange system (cf. Theorem 3 in \cite{CLLT}).
	
	These sharp inequalities of Hardy-Littlewood-Sobolev type, Stein-Weiss type, Herbst type,
and the reversed inequalities of Hardy-Littlewood-Sobolev type and Stein-Weiss type
have many important applications in partial differential equations,
geometry and quantum field theory. They often play the crucial roles in the
study of geometric elliptic equations and integral equations involving critical exponents,
such as studying sharp Sobolev inequalities in Yamabe problems, sharp logarithmic Sobolev inequalities in Ricci flows
and estimating the upper and the lower bounds of the Coulomb energy in the Thomas-Fermi model.
Naturally, we are concerned with the reversed Herbst inequality and the corresponding extremal functions.

The first major conclusion is as follows.

	\begin{theorem}\label{Rth1} (Reversed Herbst inequality)
		Assume that $n\geq 1$, $p \in (0,1)$, $q' \in (0,1)$ and $\alpha>n$ satisfy
the Sobolev-type condition
\begin{equation}\label{CRI}
\frac{1}{p}+\frac{1}{q'}-\frac{2\alpha}{q'n}=1.
\end{equation}
				Then there exists a constant $C_{n,\alpha,p,q'}>0$ such that
		\begin{equation}\label{500}
			\left|\int_{\mathbb{R}^n}\int_{\mathbb{R}^n}|x-y|^{\alpha/q'-n}|y|^{\alpha/q'}g(x)h(y)dxdy\right|
			\geq C_{n,\alpha,p,q'}\|g\|_{L^{q'}(\mathbb{R}^n)}\|h\|_{L^p(\mathbb{R}^n)}
		\end{equation}
		for any nonnegative functions $g \in L^{q'}(\mathbb{R}^n)$ and $h \in L^p(\mathbb{R}^n)$.
	\end{theorem}

\begin{remark}
(i) The reversed inequality \eqref{500} cannot be covered by \eqref{vv1}.
In fact, $\alpha \in (-n(1-r)/r,0]$, $\beta \in (-n(1-s)/s,0]$ and the
Sobolev-type condition \eqref{sbv4} imply $\alpha \in
(-n(1-r)/r,-n(1-r)/r-\lambda)$ and $\beta \in (-n(1-s)/s,-n(1-s)/s-\lambda)$.
This shows that both \eqref{vv1} with $\alpha=0$ and \eqref{vv1} with $\beta=0$
may not hold when $\max\{r,s\}<n/(n-\lambda)$.

(ii) Comparing \eqref{500} with \eqref{BI2}, we observe that their right hand sides are different.
Namely, $p$ and $q'$ in the right hand side of \eqref{500} are not H\"older conjugated.
Otherwise, $pq'<0$ and the right hand side of \eqref{500} becomes a quotient which
restricts the application of \eqref{500}.
In addition, the extremal function of \eqref{BI2} may not exist,
but we can find an extremal function of \eqref{500}
(see the following Theorem \ref{Rth2}).
\end{remark}

Write
	\begin{equation*}
		V_{\alpha}(h)(x):=\int_{\mathbb{R}^n}|x-y|^{\alpha/q'-n}h(y)|y|^{\alpha/q'}dy,
	\end{equation*}
and
\begin{equation}\label{cc1}
		C_{n,\alpha, p, q'}:=\inf\{\|V_{\alpha}(h)\|_{L^q(\mathbb{R}^n)}:h\geq 0, \|h\|_{L^p(\mathbb{R}^n)}=1\},
	\end{equation}
	where ${1}/{q}+{1}/{q'}=1$.

	\begin{theorem}\label{Rth2}
		Assume $n\geq 1$, $p\in (0,1)$, $q'\in (0,1)$ and $\alpha >n$ satisfy
\eqref{CRI}. Then		
there exists a non-negative function $h\in L^p(\mathbb{R}^n)$ such that $\|h\|_{L^p(\mathbb{R}^n)}=1$ and $\|V_{\alpha}(h)\|_{L^q(\mathbb{R}^n)}=C_{n,\alpha, p, q'}$.
	\end{theorem}
	
Here we use the scheme in \cite{LEH} to prove Theorem \ref{Rth2}.
This scheme has been successfully applied to find the extremal functions of the reversed inequalities
\eqref{vv6} and \ref{vv1} (cf. \cite{NN} and \cite{CLLT} respectively).	

\vskip 3mm
	
	To describe the extremal functions, we focus on the corresponding Euler-Lagrange system.
Minimizing the functional
	$$
		J(g,h)=\int_{\mathbb{R}^n}\int_{\mathbb{R}^n}|x-y|^{\alpha/q'-n}g(x)h(y)|y|^{\alpha/q'}dxdy
	$$
	under the constrictions $\|g\|_{L^{q'}(\mathbb{R}^n)}=\|h\|_{L^p(\mathbb{R}^n)}=1$,
we obtain
	$$
		\left\{\begin{matrix}
\displaystyle			J(g,h)g(x)^{q'-1}=\int_{\mathbb{R}^n}|x-y|^{\alpha/q'-n}h(y)|y|^{\alpha/q'}dy,	\\
\displaystyle			J(g,h)h(x)^{p-1}=\int_{\mathbb{R}^n}|x-y|^{\alpha/q'-n}g(y)|x|^{\alpha/q'}dy.
		\end{matrix}\right.
	$$
	Let $u=c_3g^{q'-1}$, $v=c_4h^{p-1}$, $-p_1=1/(q'-1)$, $-p_2=1/(p-1)$ and $\beta=\alpha/q'$.
	Now, the system above becomes
	\begin{equation}\label{j3}
		\left\{\begin{matrix}
\displaystyle			u(x)=\int_{\mathbb{R}^n}|x-y|^{\beta-n}v^{-p_2}(y)|y|^{\beta}dy,	\\
\displaystyle			v(x)=\int_{\mathbb{R}^n}|x-y|^{\beta-n}u^{-p_1}(y)|x|^{\beta}dy.
		\end{matrix}\right.
	\end{equation}
	for appropriate constants $c_3$ and $c_4$. Now, the Sobolev-type condition \eqref{CRI}
becomes		
\begin{equation}\label{105}
\frac{1}{p_1-1}+\frac{1}{p_2-1}=\frac{2\beta}{n}-1.
\end{equation}
	
	We have the nonexistence result for the positive solution of \eqref{j3}.

	\begin{theorem}\label{Rth5}
		If $\beta<n$ and $\max\{p_1, p_2\}>0$,
		\eqref{j3} has no positive solution in $L_{loc}^{\infty}(\mathbb{R}^n\setminus\{0\})$.
	\end{theorem}	

When $\beta \geq 0$, we call $(u,v)$ a pair of {\it positive entire solutions} of \eqref{j3}, if
$u,v \in L_{loc}^\infty(\mathbb{R}^n)$ solve \eqref{j3} pointwise, and $u>0$ in $\mathbb{R}^n$,
$v>0$ in $\mathbb{R}^n \setminus \{0\}$ and $v(0)=0$.
	
	\begin{theorem}\label{th3}
		Let $(u,v)$ be a pair of positive entire solutions of \eqref{j3} with $\beta>n$.
		If either $p_1>\beta/(\beta-n)$ or $2\beta/(\beta-n)<p_2<(n+\beta)/\beta$,
		we have

(i) $u(x)\simeq |x|^{\beta-n}$ and
				$v(x)\simeq |x|^{2\beta-n}$ when $|x| \to \infty$;
$v(x)\simeq |x|^{\beta}$ when $|x| \to 0$;

		(ii) $u^{-1}\in L^s(\mathbb{R}^n)$ for all $s>n/(\beta-n)$,
		$v^{-1}\in L^t(\mathbb{R}^n)$ for all $t \in (n/(2\beta-n),n/\beta)$.
		
		(iii) Furthermore, if $u$ and $v$ are differentiable,
the Sobolev-type condition \eqref{105} must hold.
	\end{theorem}
	
In fact, the assumption of differentiability in Theorem \ref{th3} is not essential.
Note that $u$ and $v$ have no singularity when $\beta>0$. Similar to the regularity arguments
in \cite{CLLT,L}, both $u$ and $v$ are differentiable. Thus,
Theorem \ref{th3} (i) and (ii) show the asymptotic behavior and the integrability
of positive entire solutions. Theorem \ref{th3} (iii) implies the Sobolev-type condition
is a necessary condition of the existence of positive entire solutions.

\vskip 5mm

Comparing with the reversed Herbst-type system \eqref{j3},	
we consider the Herbst-type system
\begin{equation}\label{j8}
		\begin{cases}
			u(x)=\displaystyle\int_{\mathbb{R}^n}\frac{v^{q_2}(y)}{|x-y|^{n-\gamma}|y|^{\gamma}}dy,	\\
			v(x)=\displaystyle\int_{\mathbb{R}^n}\frac{u^{q_1}(y)}{|x-y|^{n-\gamma}|x|^{\gamma}}dy,
		\end{cases}
	\end{equation}
where $n \geq 1$, $\gamma \in (0,n)$, $\min\{q_1,q_2\}>0$.
It is the Euler-Lagrange system related to the Herbst inequality \eqref{BI2}.
Although \eqref{BI2} may not have an extremal function, we can consider the existence of
positive super solutions of \eqref{j8}.
Positive functions $u$ and $v$ are called the positive super solutions (lower solutions) of \eqref{j8},
if $u$ and $v$ are positive in $\mathbb{R}^n \setminus \{0\}$ and satisfy
$$
\begin{cases}
\displaystyle			u(x) \geq (\leq) \ c_1\int_{\mathbb{R}^n}\frac{v^{q_2}(y)}{|x-y|^{n-\gamma}|y|^{\gamma}}dy,	\\
\displaystyle			v(x) \geq (\leq) \ c_2\int_{\mathbb{R}^n}\frac{u^{q_1}(y)}{|x-y|^{n-\gamma}|x|^{\gamma}}dy,
		\end{cases}
$$
for some positive constants $c_1$ and $c_2$.

Recall several existence results of \eqref{889}. Now,
the Sobolev-type condition \eqref{sbv2} becomes
$$
\frac{1}{r+1}+\frac{1}{s+1}=\frac{\alpha+\lambda+\beta}{n}.
$$
When $\alpha=\beta=0$, it may be a critical condition describing the existence/nonexistence of positive regular
solutions of \eqref{889}. In fact, Theorem 5.1 in \cite{CDM} shows the nonexistence of positive radial
solutions in the subcritical case. Such a critical condition also plays the key role in the study of
the Lane-Emden system of partial differential equations. It is associated with the Lane-Emden conjecture.
In addition, \cite{CDM} also gave
another condition (the Serrin-type condition) which can be used to describe the
existence/nonexistence of positive super solutions of \eqref{889} (see also \cite{DM,LeiLi,LL}).

For \eqref{j8}, the Sobolev-type condition \eqref{sbv5} becomes
    \begin{equation}\label{serrin}
		\frac{1}{q_1+1}+\frac{1}{q_2+1}=1.
    \end{equation}
    Clearly, it is equivalent to $q_1q_2=1$.

\begin{theorem}\label{th10} Let $n \geq 1$, $\gamma \in (0,n)$ and $\min\{q_1,q_2\}>0$.

		(i) If one of the three conditions $0<q_1q_2<1$, $0 < q_1 \leq \gamma/(n-\gamma)$ and
$q_2 \geq (n-\gamma)/\gamma$ holds, \eqref{j8} has no positive super solution in
$L_{loc}^{\infty}(\mathbb{R}^n\setminus\{0\})$ for any positive constants $c_1$ and $c_2$.
		
(ii) If $q_1q_2=1$, and $q_1 > \gamma/(n-\gamma)$ (or $0 <q_2< (n-\gamma)/\gamma$), \eqref{j8} has positive
radially symmetric and decreasing super solutions (and lower solutions) in
$L_{loc}^{\infty}(\mathbb{R}^n\setminus\{0\})$.
	\end{theorem}

When $q_1q_2>1$,
we have the following two nonexistence results.

\begin{theorem}\label{th1.6}
Let $n \geq 1$, $\gamma \in (0,n)$ and $\min\{q_1,q_2\}>0$.
If $q_1q_2>1$, then \eqref{j8} has no radially symmetric and decreasing
positive super solution in $L_{loc}^{\infty}(\mathbb{R}^n\setminus\{0\})$.
\end{theorem}

\begin{theorem}\label{th1.7}
Let $n \geq 1$, $\gamma \in (1,n)$ and $\min\{q_1,q_2\}>0$, and
positive functions $u$ and $v$ be differentiable in
$\mathbb{R}^n \setminus \{0\}$ and solve \eqref{j8}.
Assume either

(i) $(u,v) \in [L^{q_1+1}(\mathbb{R}^n)
\cap L^{s}(\mathbb{R}^n)] \times L^{q_2+1}(\mathbb{R}^n)$
for some $s>nq_1/(\gamma-1)$, or

(ii) $(u,v) \in L^{q_1+1}(\mathbb{R}^n)
\times [L^{q_2+1}(\mathbb{R}^n) \cap L^{s}(\mathbb{R}^n)]$
for some $s>nq_2/(\gamma-1)$.

Then $q_1q_2=1$.
\end{theorem}
	
The radial symmetry and the integrability of positive solutions
of \eqref{j8} seem difficult to be obtained as in
the study of \eqref{889}, because the Herbst inequality \eqref{BT2}
(or \eqref{BT1}) does not work in lifting the regularity when we use the contraction argument
as in \cite{CJLL,CLO-CPDE,JL,JL2}.
We expect to obtain the radial symmetry of those solutions by the same
argument in \cite{LZ} where a modified version of the method of moving planes in integral form
(introduced in \cite{DGZ}) was applied.
We also expect $(u,v) \in L^{s_1}(\mathbb{R}^n) \times L^{s_2}(\mathbb{R}^n)$
for all $s_i >nq_i/(\gamma-1)$ $(i=1,2)$
when $(u,v) \in L^{q_1+1}(\mathbb{R}^n) \times L^{q_2+1}(\mathbb{R}^n)$,
which was proved for the Hardy-Littlewood-Sobolev system in \cite{CJLL,JL2}.
Once this lifting integrability holds, Theorem \ref{th1.7} shows
that there is no positive differentiable solution in
$L^{q_1+1}(\mathbb{R}^n) \times L^{q_2+1}(\mathbb{R}^n)$ when
$q_1q_2>1$.

\section{Proof of Theorem \ref{Rth1}}
	
		The following three lemmas will be used here.
	
	\begin{lemma}\label{Rle1} (Reversed Young's inequality, Lemma 2.1 in \cite{DZ}, Lemma 10 in \cite{CLLT})
		Let $G$ be a locally compact group with left Haar measure $\mu$ that satisfies
		$\mu(A)=\mu(A^{-1})$ for all measurable sets $A\subseteq G$. Assume that $p, q$ and $s$ satisfy
		$$
		0<p<1, \quad \max\{q,s\}<0, \quad \frac{1}{q} +1=\frac{1}{p}+\frac{1}{s}.
		$$
		Then for any non-negative $g\in L^p(G,\mu)$ and $h\in L^s(G,\mu)$, we have
		$$
		\|g\ast h\|_{L^q(G,\mu)}\geq \|g\|_{L^p(G,\mu)}\|h\|_{L^s(G,\mu)},
		$$
		where
		$(g\ast h)(x)=\int_Gg(y)h(y^{-1}x)d\mu(y).$
		Here $y^{-1}$ denotes the inverse of $y$ in the group $G$.
	\end{lemma}

By the same proof of Lemma 1 in \cite{NN}
(or Lemma 11 in \cite{CLLT}), we also have the following lemma.

\begin{lemma}\label{lem2.3}
		
		(i) For $p,q\in (0,1)$, $\alpha>n$ and any non-negative functions $g(x)$ and $h(x)$ defined on $\mathbb{R}^n$, we have
		\begin{equation}\label{5000}
			\int_{\mathbb{R}^n}\int_{\mathbb{R}^n}|x-y|^{\alpha/q'-n}g(x)h(y)|y|^{\alpha/q'}dxdy
			\geq \int_{\mathbb{R}^n}\int_{\mathbb{R}^n}|x-y|^{\alpha/q'-n}g^*(x)h^*(y)|y|^{\alpha/q'}dxdy,
		\end{equation}
		where $g^*$ and $h^*$ are rearrangements of $g$ and $h$.
		
		(ii) If $h$ is radically decreasing, then $V_{\alpha}(h)$ is radically increasing.
		
		(iii) For any non-negative function $h\in L^p(\mathbb{R}^n)$, there holds
		$$
		\|V_{\alpha}(h)\|_{L^q(\mathbb{R}^n)}\geq \|V_{\alpha}(h^*)\|_{L^q(\mathbb{R}^n)}.
		$$
	\end{lemma}

	\begin{lemma}\label{Rle2} (\cite{CLLT}, Lemma 12)
		Let $x\in S^{n-1}=\{x\in \mathbb{R}^n:|x|=1\}$, and denote
		$$
		I(x)=\int_{S^{n-1}}\varphi (x\cdot y)dy.
		$$
		Then $I(x)$ is a constant independent of $x$ and
		$$
		I(x)=\omega_{n-2}\int_{-1}^{1}\varphi(t)(1-t^2)^{\frac{n-3}{2}}dt,
		$$
		where $\omega_{n-2}$ denotes the area of $S^{n-2}$.
	\end{lemma}

	\textbf{Proof of Theorem \ref{Rth1}.}
	
\begin{proof}
		Set
		$$
		\tilde{h}(y):=|y|^{\alpha/q'}h(y), \quad
(T\tilde{h})(x):=\int_{\mathbb{R}^n}\tilde{h}(y)|x-y|^{\alpha/q'-n}dy.
		$$	
	
Now, we prove Theorem \ref{Rth1} in two cases.

		{\it Case 1:} $n=1$.

		By the reversed H\"older inequality, we have
		$$
		\int_{\mathbb{R}}\int_{\mathbb{R}}|x-y|^{\alpha/q'-n}|y|^{\alpha/q'}g(x)h(y)dxdy
		\geq \|T\tilde{h}\|_{L^q(\mathbb{R})}\|g\|_{L^{q'}(\mathbb{R})}.
		$$
Once there holds
\begin{equation}\label{e19}
			\|T\tilde{h}\|_{L^q(\mathbb{R})}\geq C_{1,\alpha, p, q'}\||x|^{-\alpha/q'}\tilde{h}\|_{L^p(\mathbb{R})},
		\end{equation}
		We can deduce \eqref{500} immediately.

		According to the left Haar measure, \eqref{e19} is equivalent to the following inequality	
		\begin{equation}\label{hhn}
\||x|^{1/q}T\tilde{h}\|_{L^q(\mu)}\geq C_{1,\alpha, p, q'}\||x|^{-\alpha/q'+1/p}\tilde{h}\|_{L^p(\mu)},
\end{equation}
		where $d\mu(x)=|x|^{-1}dx$.
		Since \eqref{CRI} is reduced to
		\begin{equation}\label{k1}
			\frac{1}{p}+\frac{1}{q'}-\frac{2\alpha}{q'}=1,
		\end{equation}
		we obtain
		$$		
|x|^{1/q}T\tilde{h}(x)
=\int_{-\infty}^{\infty}\frac{|x|^{1/q}\tilde{h}(y)|y|^{-\alpha/q'+1/p}}
{|y|^{1/q}|1-\frac{x}{y}|^{1-\alpha/q'}}\frac{dy}{|y|}
		=(w\ast f)(x),
		$$
		where $w(x)=|x|^{-\alpha/q'+1/p}\tilde{h}(x)$ and $f(x)=|x|^{1/q}|1-x|^{\alpha/q'-1}$.		
		Applying Lemma \ref{Rle1}, we have
		\begin{equation}\label{501}
			\||x|^{1/q}T\tilde{h}\|_{L^q(\mu)}\geq \||x|^{-\alpha/q'+1/p}\tilde{h}\|_{L^p(\mu)}\|f\|_{L^s(\mu)},
		\end{equation}
		where $1/p+1/s=1+1/q$ and $s<0$.
		
		We claim that  $f\in L^s(\mu)$.
				
		Since the defects of the improper integral above may happen at $0$, $1$ and $\infty$,
		for some $\delta \in (0,1)$, there holds
		$$
		\begin{aligned}
			\int_{0}^{\infty}r^{s/q}|1-r|^{(\alpha/q'-1)s}\frac{dr}{r}
			=&\int_{0}^{1-\delta}r^{s/q}|1-r|^{(\alpha/q'-1)s}\frac{dr}{r}\\
			+&\int_{1-\delta}^{1+\delta}r^{s/q}|1-r|^{(\alpha/q'-1)s}\frac{dr}{r}
			+\int_{1+\delta}^{\infty}r^{s/q}|1-r|^{(\alpha/q'-1)s}\frac{dr}{r}\\
			:=&D_1+D_2+D_3.
		\end{aligned}
		$$
		Note that
		$$
		r^{s/q}|1-r|^{(\alpha/q'-1)s} \simeq Cr^{s/q}, \quad as \; r \to 0,
		$$
		and
		$$
		r^{s/q}|1-r|^{(\alpha/q'-1)s} \simeq C(1-r)^{(1/q+\alpha/q'-1)s}, \quad as \;  r\to \infty.
		$$
		By virtue of $s<0$, $q<0$ and $\alpha>1$, we get
		$$
\frac{s}{q}>0 \quad and \quad \left(\frac{1}{q}+\frac{\alpha}{q'}-1\right)s<0,
$$
		which imply $D_1<\infty$ and $D_3<\infty$.		
		In addition,
		$$
		D_2\leq C \int_0^{\delta}t^{(\alpha/q'-1)s}dt <\infty,
		$$
		because $1/p+1/s=1+1/q$ and \eqref{k1} imply $(\alpha/q'-1)s+1>0$.
The estimates of $D_i(i=1,2,3)$ lead to $f \in L^s(\mu)$.
		Thus, \eqref{501} implies \eqref{hhn}, and hence \eqref{e19} is proved.

		{\it Case 2:} $n\geq 2$.
		
		According to the properties of rearrangement, $h^*$ is radially symmetric and
		decreasing, and
		\begin{equation}\label{k3}
			\|h^*\|_{L^p(\mathbb{R}^n)}=\|h\|_{L^p(\mathbb{R}^n)} \quad  and \quad
			\|g^*\|_{L^{q'}(\mathbb{R}^n)}=\|g\|_{L^{q'}(\mathbb{R}^n)}.
		\end{equation}
		Clearly, $\tilde{h^*}$ is also radially symmetric. 		
		Therefore, using Lemma \ref{Rle2} we obtain
		\begin{equation*}
			\begin{aligned}
				T\tilde{h^*}(x)=&\int_{\mathbb{R}^n}\tilde{h^*}(y)|x-y|^{\alpha/q'-n}dy\\
				=&\int_{0}^{\infty}\int_{S^{n-1}}\tilde{h^*}(r)|x-ry'|^{\alpha/q'-n}r^{n-1}dy'dr\\
				=&\int_{0}^{\infty}\int_{S^{n-1}}\tilde{h^*}(r)
(r^2-2r|x| x'\cdot y'+{|x|}^2)^{\frac{\alpha/q'-n}{2}}r^{n-1}dy'dr\\
				=&\omega_{n-2}\int_{0}^{\infty}\int_{-1}^{1}(1-t^2)^{\frac{n-3}{2}}(r^2-2r|x| t+{|x|}^2)^{\frac{\alpha/q'-n}{2}}\tilde{h^*}(r)r^{n-1}dt dr\\
				=&\omega_{n-2}\int_{0}^{\infty}\int_{-1}^{1}(1-t^2)^{\frac{n-3}{2}}\left(1-2\left(\frac{|x|}{r}\right)t
+\left(\frac{|x|}{r}\right)^2\right)^{\frac{\alpha/q'-n}{2}} \tilde{h^*}(r)r^{\alpha/q'-1}dtdr.
			\end{aligned}
		\end{equation*}
		Write
		$$
		I_{\alpha}(\rho):=\int_{-1}^{1}(1-t^2)^{\frac{n-3}{2}}(1-2\rho t+\rho^2)^{\frac{\alpha/q'-n}{2}}dt, \quad \rho>0.
		$$
		Obviously, $I_{\alpha}(\rho)$ is well defined and continuous for any $\rho>0$, and
$$
		|x|^{n/q}T\tilde{h^*}(x)=\omega_{n-2}\int_{0}^{\infty}\tilde{h^*}(r)r^{\alpha/q'+n/q}
		\frac{|x|^{n/q} }{r^{n/q}}I_{\alpha}\left(\frac{|x| }{r}\right)\frac{dr}{r}.
		$$
Applying Lemma \ref{Rle1} and \eqref{CRI}, we get
		\begin{equation}\label{k2}
			\begin{aligned}
				\||x|^{n/q}T\tilde{h^*}(x)\|_{L^q(\mu)}
				\geq &\omega_{n-2}\||x|^{\alpha/q'+n/q}\tilde{h^*}\|_{L^p(\mu)}\||x|^{n/q}I_{\alpha}(x)\|_{L^s(\mu)}\\
				=&\omega_{n-2} \||x|^{-\alpha/q'+n/p}\tilde{h^*} \|_{L^p(\mu)}\||x|^{n/q}I_{\alpha}(x)\|_{L^s(\mu)}.
			\end{aligned}
		\end{equation}

		Noticing
		$$
		\frac{I_{\alpha}(r)}{r^{\alpha/q'-n}}\simeq \int_{-1}^{1}(1-t^2)^{\frac{n-3}{2}}dt=C, \quad as\;\; r\to \infty,
		$$
		we have		
\begin{equation}\label{ub1}
r^{ns/q}I_{\alpha}^s(r) \simeq Cr^{ns/q+\alpha s/q'-ns}, \quad as \; r\to \infty.
\end{equation}		
Similarly,
\begin{equation}\label{ub2}
r^{ns/q}I_{\alpha}^s(r) \simeq Cr^{ns/q},\quad as \; r \to 0.
\end{equation}		
By virtue of $s<0$, $q<0$ and $\alpha>n$, there hold
		$$
\frac{ns}{q}>0 \quad and \quad \frac{ns}{q}+\frac{\alpha s}{q'}-ns<0.
$$
This result, together with \eqref{ub1} and \eqref{ub2}, implies
$r^{n/q}I_{\alpha}(r)\in {L^s(\mu)}$.	
		Inserting this result into \eqref{k2}, we get
		$$
		\|T\tilde{h^*}(x)\|_{L^q(\mathbb{R}^n)} \geq C\||x^{-\alpha/q'}\tilde{h^*}|\|_{L^p(\mathbb{R}^n)}.
		$$
Therefore, by the reversed H\"{o}lder inequality and \eqref{k3}, we can deduce
\begin{equation}\label{ub3}
\begin{aligned}
&\int_{\mathbb{R}^n}\int_{\mathbb{R}^n}|x-y|^{\alpha/q'-n}g^*(x)h^*(y)|y|^{\alpha/q'}dxdy\\
\geq &\|T\tilde{h^*}(x)\|_{L^q(\mathbb{R}^n)} \|g^*\|_{L^{q'}(\mathbb{R}^n)}
\geq C\||x^{-\alpha/q'}\tilde{h^*}|\|_{L^p(\mathbb{R}^n)} \|g^*\|_{L^{q'}(\mathbb{R}^n)}\\
=&C\|{h^*}|\|_{L^p(\mathbb{R}^n)} \|g^*\|_{L^{q'}(\mathbb{R}^n)}
=C\|{h}|\|_{L^p(\mathbb{R}^n)} \|g\|_{L^{q'}(\mathbb{R}^n)}.
\end{aligned}
\end{equation}
		Combining this result with \eqref{5000}, we obtain \eqref{500}.
		This completes the proof of Theorem \ref{Rth1}.
	\end{proof}

\section{Proof of Theorem \ref{Rth2}}

We need the following lemma.

\begin{lemma}\label{Rlem2} (\cite{DZ}, Theorem 3.2)
	Suppose that $h \in L^p(\mathbb{R}^n)$ is non-negative, radically symmetric, and satisfies
	$h(|x|)\leq \varepsilon |x|^{-n/p}$ for all $|x|>0$.
Then for any $0<t<p$, there exists a constant $C>0$ independent of $h$ and $\varepsilon$ such that
	$$
	\|V_{\alpha}(h)\|_{L^q(\mathbb{R}^n)}\geq C\varepsilon^{1-p/t}\|h\|_{L^p(\mathbb{R}^n)}^{p/t}.
	$$
\end{lemma}

\textbf{Proof of Theorem \ref{Rth2}.}

\begin{proof}
		We divide the proof into four steps.
		
		{\it Step 1. Minimizing sequence $\{H_j\}_j$ of \eqref{cc1} and its limit.}
		
		Let $\{h_j\}_j$ be a minimizing sequence for \eqref{cc1}.
According to (iii) in Lemma \ref{lem2.3} and \eqref{ub3}, $\{h_j^*\}_j$ is also the minimizing sequence.
Therefore, we can assume at the beginning that $\{h_j\}_j$ is a non-negative radically
symmetric and decreasing sequence. Now,
		\begin{equation}\label{k4}
			\lim_{j\to \infty}\|V_{\alpha}(h_j)\|_{L^q(\mathbb{R}^n)}= C_{n,\alpha,p,q'}.
		\end{equation}

		By the monotonicity of $\{h_j\}_j$, from $\|h_j\|_{L^p(\mathbb{R}^n)}=1$ it follows
		$$
		\frac{\omega_{n-1}}{n}h_j^p(R)R^n\leq \omega_{n-1}\int_{0}^{R}h_j^p(r)r^{n-1}dr\leq \int_{\mathbb{R}^n}h_j^p(x)dx=1
		$$
		for any $R>0$. Thus, there holds
		$$
		0\leq h_j(R)\leq CR^{-n/p}, \quad \forall \ R>0,
		$$
		where $C>0$ is a constant independent of $j$ and $R$.		
		Therefore,
		$$
		a_j:=\sup_{r>0}r^{n/p}h_j(r)\leq C.
		$$
		Combining $\|h_j\|_{L^p(\mathbb{R}^n)}=1$ and \eqref{k4},
		we derive from Lemma \ref{Rlem2} that $a_j\geq 2c_0$ for some constant $c_0>0$,
which implies that we can choose $\lambda_j>0$ such that $\lambda_j^{n/p}h_j(\lambda_j)>c_0$.
Set
		$$
		H_j(x):=\lambda_j^{n/p}h_j(\lambda_jx).
		$$		
Then for any $j$, $H_j$ is also nonnegative, radially symmetric and decreasing.
		By simple calculations, we find that
		$$
		\|H_j(x)\|_{L^p(\mathbb{R}^n)}=\|h_j(x)\|_{L^p(\mathbb{R}^n)}=1 \quad
		and \quad
		\|V_{\alpha}(H_j)\|_{L^q(\mathbb{R}^n)}=\|V_{\alpha}(h_j)\|_{L^q(\mathbb{R}^n)}.
		$$
		Therefore, $\{H_j(x)\}_j$ is also a minimizing sequence satisfying $H_j(1)>c_0$.
		By an analogous argument in \cite{LEH}, using the Helly theorem (cf. p. 89 in \cite{LEHL}),
we can find a subsequence of $\{H_j\}_j$ denoted by itself such that
$$
H_j(x) \to h(x) \quad a.e. \quad in \quad \mathbb{R}^n, \quad (j \to \infty).
$$
Clearly, $h$ is nonnegative. Applying the Fatou lemma and $\|H_j(x)\|_{L^p(\mathbb{R}^n)}=1$,
		we know $h \in L^p(\mathbb{R}^n)$.		
In the following, we will prove that the limit $h$ is a minimizer of \eqref{cc1}.

{\it Step 2. We claim}
\begin{equation}\label{k6}
			\lim_{j\to \infty}\left(\int_{\mathbb{R}^n}[V_{\alpha}(H_j)]^q(x)dx\right)^{1/q}
			=\left(\int_{\mathbb{R}^n}\lim_{j\to \infty}[V_{\alpha}(H_j)]^q(x)dx\right)^{1/q}.
		\end{equation}
		
		In fact, Lemma \ref{lem2.3} (ii) shows that, for any $j$, $V_\alpha (H_j)(x)$ is radially symmetric and increasing.
In addition, there holds
		\begin{equation}\label{507}
			V_\alpha (H_j)(x)\geq c_0\int_{|y|\leq 1}|x-y|^{\alpha/q'-n}|y|^{\alpha/q'}dy\geq C(1+|x|^{\alpha/q'-n})
		\end{equation}
		for all $x\in \mathbb{R}^n$. Here positive constant $C$ is independent of $j$.
		
		In view of \eqref{k4}, we can find a constant $C>1$ such that
$C^{-1} \leq \|V_{\alpha}(H_j)\|_{L^q(\mathbb{R}^n)} \leq C$ for any $j$.
Consequently, for any $R>0$, we obtain that
		$$
		[V_{\alpha}(H_j)]^q(R)R^n
\leq C\int_{|x| \leq R}[V_{\alpha}(H_j)]^q(x)dx
\leq C\int_{\mathbb{R}^n}[V_{\alpha}(H_j)]^q(x)\leq C,
		$$
		which implies
		$$
		0\leq [V_{\alpha}(H_j)]^{-1}(R)\leq CR^{n/q}, \quad \forall \ R>0.
		$$
		Note that $[V_{\alpha}(H_j)]^{-1}$ is radically symmetric and decreasing.
Applying the Helly theorem, we can find a subsequence of $\{[V_{\alpha}(H_j)]^{-1}\}_j$
denoted by itself such that
$$
[V_{\alpha}(H_j)]^{-1}(x) \to k(x) \quad a.e. \quad in \quad \mathbb{R}^n, \quad (j \to \infty).
$$	
In addition, $(1+|x|^{\alpha/q'-n})^q \in L^1(\mathbb{R}^n)$ (by virtue of $\alpha>n$).
Using \eqref{507} and the dominated convergence theorem, we obtain \eqref{k6}.

{\it Step 3. We claim}	
\begin{equation}\label{k5}
			\|h\|_{L^p(\mathbb{R}^n)}=1.
		\end{equation}

In fact, \eqref{k4} and \eqref{k6} imply $meas\{x \in \mathbb{R}^n; k(x)>0\}>0$.
Thus, we can find two distinct points $x_1,x_2 \in \mathbb{R}^n$ satisfying
$[V_{\alpha}(H_j)](x_m) \to k^{-1}(x_m) <\infty$ $(m=1,2)$, which implies
\begin{equation}\label{ub6}
[V_{\alpha}(H_j)](x_m) \leq C \quad (m=1,2).
\end{equation}
Here $C>0$ is independent of $j$. Therefore,
by $|a+b|^{\alpha/q'-n} \leq C(|a|^{\alpha/q'-n}+|b|^{\alpha/q'-n})$, it follows
$$
|x_1-x_2|^{\alpha/q'-n}\int_{\mathbb{R}^n}H_j(y)|y|^{\alpha/q'}dy
\leq C\sum_{m=1}^2[V_{\alpha}(H_j)](x_m) \leq C.
$$	
Applying this result and the reversed H\"{o}lder inequality, we obtain that for
all $R>2|x_1|$,
		$$
\left(\int_{|y|\leq R}H_j^{\tau}(y)dy\right)^{1/\tau}
\left(\int_{|y|\leq R}|y|^{\alpha\tau'/q'}dy\right)^{1/\tau'}
\leq \int_{|y|\leq R}H_j(y)|y|^{\alpha/q'}dy \leq C.		
		$$
		Here $\tau \in (0,1)$ and $1/\tau+1/\tau'=1$. In view of $\alpha>n$,
from \eqref{CRI} it follows $1+\alpha/(nq')<1/p$. Therefore,
we can choose $1/\tau \in (1+\alpha/(nq'),1/p)$ which implies $\tau'>-nq'/\alpha$.
Thus, there exists some constant $C(R)$ (independent of $j$) such that
$$
			\int_{|y|\leq R}H_j^{\tau}(y)dy \leq C(R).
$$
Therefore, by the monotonicity of $H_j$, there holds
$$
H_j^{\tau}(x)|x|^n \leq C\int_{B_{|x|}(0)}H_j^{\tau}(y)dy \leq C(R),
\quad \forall x \in B_R(0),
$$
which implies $H_j^p(x) \leq C|x|^{-np/\tau}$ for all $x \in B_R(0)$. Here $C>0$ is
independent of $j$ and $x$. In view of $|x|^{-np/\tau} \in L^1(B_R(0))$, using
the dominated convergence theorem, we get
\begin{equation}\label{508}		
\lim_{j\to \infty}\int_{|y|\leq R}H_j^p(y)dy=\int_{|y|\leq R}h^p(y)dy.
\end{equation}

On the other hand, in view of $R>2|x_1|$, by the monotonicity of $H_j$, we have
		$$
			\int_{\mathbb{R}^n}|x_1-y|^{\frac{\alpha}{q'}-n}H_j(y)|y|^{\frac{\alpha}{q'}}dy
			\geq \int_{3R/4 \leq |y| \leq R}|x_1-y|^{\frac{\alpha}{q'}-n}H_j(y)|y|^{\frac{\alpha}{q'}}dy
			\geq CR^{\frac{2\alpha}{q'}}H_j(R),
		$$
where $C>0$ is independent of $j$ and $R$.
		This, together with \eqref{ub6}, implies $H_j(R) \leq CR^{-2\alpha/q'}$ for sufficiently large $R$.
Therefore, in view of $n<2p\alpha/q'$ which is implied by \eqref{CRI}, we get
		\begin{equation}\label{509}
			\lim_{R\to \infty}\lim_{j\to \infty}\int_{|y|\geq R}H_j^p(y)dy=0.
		\end{equation}
		In view of \eqref{508} and \eqref{509}, we obtain that
		$$
		\begin{aligned}
			\lim_{j\to \infty}\int_{\mathbb{R}^n}H_j^p(y)dy
			=&\lim_{R\to \infty}\lim_{j\to \infty}\int_{|y|\leq R}H_j^p(y)dy
+\lim_{R\to \infty}\lim_{j\to \infty}\int_{|y|\geq R}H_j^p(y)dy\\
			=&\lim_{R\to \infty}\int_{|y|\leq R}h^p(y)dy
			=\int_{\mathbb{R}^n}h^p(y)dy,
		\end{aligned}
		$$
		which implies \eqref{k5}.
		
{\it Step 4.}
		Applying the Fatou lemma and noting that $q<0$, we have
		\begin{equation}\label{k7}
			\lim_{j\to \infty} [V_{\alpha}(H_j)]^q=[ \lim_{j\to \infty}V_{\alpha}(H_j) ]^q
\leq [V_{\alpha}(h)]^q.
		\end{equation}

		Since $H_j$ is the minimizing sequence, by \eqref{k6}, \eqref{k7}, \eqref{cc1} and \eqref{k5}, we get
		$$
			\begin{aligned}
				C_{n,\alpha,p,q'}=&\lim_{j\to \infty}\left(\int_{\mathbb{R}^n}[V_{\alpha}(H_j)]^qdx\right)^{1/q}
				=\left(\int_{\mathbb{R}^n}\lim_{j\to \infty}[V_{\alpha}(H_j)]^qdx\right)^{1/q}\\
				\geq &\left(\int_{\mathbb{R}^n}[V_{\alpha}(h)]^qdx\right)^{1/q}
				\geq C_{n,\alpha,p,q'}.
			\end{aligned}
	    $$
		 Thus, $h$ is a minimizer of \eqref{cc1}. The proof of Theorem \ref{Rth2} is complete.
	\end{proof}

\section{Proof of Theorem \ref{Rth5}}

	\begin{proof}
		Assuming that $u,v \in L_{loc}^{\infty}(\mathbb{R}^n \setminus \{0\})$ are
positive solutions of \eqref{j3}, we will deduce a contradiction.
		
		\textit{Case 1:} $\beta \leq 0$.
		
		Let $1/2<|x|<1$. When $y \in B_{|x|/2}(x)$, we have $1/4 \leq |y| \leq 3/2$. Therefore, $\beta \leq 0$ implies
		$$
		u(x) \geq \left(\min_{1/4 \leq |y| \leq 3/2}v^{-p_2}(y)\right)\left(\frac{3}{2}\right)^\beta
\int_{B_{|x|/2}(x)}|x-y|^{\beta-n}dy=\infty.
		$$
		Therefore, $u \notin L_{loc}^{\infty}(\mathbb{R}^n \setminus \{0\})$. It is impossible.
		
		\textit{Case 2:} $0<\beta<n$.
		
		By \eqref{j3} and Lemma 3.11.3 in \cite{Z}, there exists $C>0$ such that for any $\delta>1$,
		\begin{equation}\label{a1}
			\begin{aligned}
				\frac{1}{|B_{\delta}(x_0)|}\int_{B_{\delta}(x_0)}u(x)dx
				=&\int_{\mathbb{R}^n}\left(\frac{1}{|B_{\delta}(x_0)|}
\int_{B_{\delta}(x_0)}|x-y|^{\beta-n}dx\right)
				v^{-p_2}(y)|y|^{\beta}dy\\
				\leq & C\int_{\mathbb{R}^n}|x_0-y|^{\beta-n}v^{-p_2}(y)|y|^{\beta}dy=Cu(x_0).
			\end{aligned}
		\end{equation}
		Here $x_0 \neq 0$.		
		Applying the H\"{o}lder inequality, we obtain that
		\begin{equation}\label{a2}
			\begin{aligned}
				 1=&\frac{1}{|B_{\delta}(x_0)|}\int_{B_{\delta}(x_0)}u^{-\frac{p_1}{p_1+1}}(x)u^{\frac{p_1}{p_1+1}}(x)dx\\
				\leq & \left(\frac{1}{|B_{\delta}(x_0)|}\int_{B_{\delta}(x_0)}u^{-p_1}(x)dx\right)^{\frac{1}{p_1+1}}
				\cdot
				\left(\frac{1}{|B_{\delta}(x_0)|}\int_{B_{\delta}(x_0)}u(x)dx\right)^{\frac{p_1}{p_1+1}}.
			\end{aligned}
		\end{equation}
		By \eqref{a1} and \eqref{a2}, there holds
		\begin{equation}\label{a3}
			C^{-p_1}u^{-p_1}(x_0) \leq \left(\frac{1}{|B_{\delta}(x_0)|}\int_{B_{\delta}(x_0)}u(x)dx\right)^{-p_1}
			\leq \frac{1}{|B_{\delta}(x_0)|}\int_{B_{\delta}(x_0)}u^{-p_1}(x)dx.
		\end{equation}
		In view of $0<\beta<n$,
		we have $\delta^{\beta-n}<|x-x_0|^{\beta-n}$ when $|x-x_0|<\delta$.
		Therefore, multiplying \eqref{a3} by $\delta^{\beta}|x_0|^{\beta}$ yields
		$$
		C^{-p_1}\delta^{\beta}u^{-p_1}(x_0)|x_0|^{\beta}
		\leq \frac{1}{|B_1(x_0)|}\int_{B_{\delta}(x_0)}|x_0-x|^{\beta-n} u^{-p_1}(x)|x_0|^{\beta}dx \leq \frac{v(x_0)}{|B_1(x_0)|}.
		$$
		Noting $\beta>0$, and letting $\delta\to \infty$, we have $v(x_0) = \infty$ which
contradicts with $v \in L_{loc}^{\infty}(\mathbb{R}^n \setminus \{0\})$.

Combining case 1 and case 2, we complete the proof of Theorem \ref{Rth5}.
	\end{proof}
	
	\section{Proof of Theorem \ref{th3}}

	\begin{prop}\label{prop} (Asymptotic rates)
Let $(u,v)$ be a pair of positive entire solutions of \eqref{j3} with $\beta>n$.
		Then $p_1>{\beta}/(\beta-n)$ is equivalent to $2\beta/(2\beta-n)<p_2<(n+\beta)/\beta$.
		Furthermore, if $p_1>\beta/(\beta-n)$ or $2\beta/(\beta-n)<p_2<(n+\beta)/\beta$, then
		\begin{equation*}%\label{401}
				u(x)\simeq |x|^{\beta-n}, \ \ \
				v(x)\simeq |x|^{2\beta-n}, \ \ \  when \ \ |x|\to \infty,
		\end{equation*}
and
\begin{equation*}%\label{4012}				
				v(x)\simeq |x|^{\beta}, \ \ \  when \ \ |x|\to 0.
		\end{equation*}
	\end{prop}

	\begin{proof}
		{\it Step 1.} Assume
		\begin{equation}\label{380}
			p_1>\frac{\beta}{\beta-n}.
		\end{equation}
Then we will prove $2\beta/(2\beta-n)<p_2<(n+\beta)/\beta$,
$v(x)\simeq |x|^{\beta}$ when $|x|\to 0$, and $v(x)\simeq |x|^{2\beta-n}$ when $|x|\to \infty$.

		(1) Let $|x| \geq R > 2$. Clearly,
		\begin{equation}\label{381}
			\begin{aligned}
				u(x)\geq &\int_{B_1(0) \setminus B_{1/2}(0)}\frac{|x-y|^{\beta-n}|y|^{\beta}}{v^{p_2}(y)}dy\\
				\geq &c\left(\min_{1/2 \leq |y| \leq 1}v^{-p_2}(y)\right)\int_{B_1(0) \setminus B_{1/2}(0)}|x-y|^{\beta-n}|y|^{\beta}dy
				\geq c|x|^{\beta-n}.
			\end{aligned}
		\end{equation}
		
		Similarly, we also have
		\begin{equation}\label{382}			
				v(x)\geq \int_{B_1(0)}\frac{|x-y|^{\beta-n}|x|^{\beta}}{u^{p_1}(y)}dy
				\geq c\left(\min_{|y| \leq 1}u^{-p_1}(y)\right)\int_{B_1(0)}|x-y|^{\beta-n}|x|^{\beta}dy
				\geq c|x|^{2\beta-n},
		\end{equation}
		and
		\begin{equation}\label{383}
							v_1(x):=\int_{B_R(0)}\frac{|x-y|^{\beta-n}|x|^{\beta}}{u^{p_1}(y)}dy
				\leq C\left(\max_{|y| \leq R}u^{-p_1}(y)\right)|x|^{2\beta-n}
				\leq C|x|^{2\beta-n}.			
		\end{equation}
		
		When $y\in B_{2|x|}(0)$, $|x-y|\leq |x|+|y|\leq 3|x|$.
		Thus, combining \eqref{380} and \eqref{381}, we have
		\begin{equation}\label{384}			
				v_2(x):=\int_{B_{2|x|}(0)\setminus B_R(0)}
				\frac{|x-y|^{\beta-n}|x|^{\beta}}{u^{p_1}(y)}dy
				\leq C\int_{B_{2|x|}(0)\setminus B_R(0)}
				\frac{|x|^{2\beta-n}dy}{|y|^{(\beta-n)p_1}}
				\leq C|x|^{2\beta-n}.
		\end{equation}
		
		When $y\in \mathbb{R}^n\setminus B_{2|x|}(0)$, $|x-y|\leq |x|+|y|\leq 3|y|/2$.
Combining \eqref{380} and \eqref{381} again, we get
		\begin{equation}\label{385}			
				v_3(x):=\int_{\mathbb{R}^n\setminus B_{2|x|}(0)}
				\frac{|x-y|^{\beta-n}|x|^{\beta}}{u^{p_1}(y)}dy
				\leq C\int_{\mathbb{R}^n\setminus B_{2|x|}(0)}
				\frac{|y|^{\beta-n}|x|^{\beta}dy}{|y|^{(\beta-n)p_1}}		
		\leq C|x|^{2\beta-(\beta-n)p_1}.
		\end{equation}
		Clearly, \eqref{380} implies $2\beta-(\beta-n)p_1<2\beta-n$.
		Therefore, estimates of \eqref{383}-\eqref{385} lead to
		\begin{equation}\label{386}
			v(x)\leq C|x|^{2\beta-n} \ \ when \ \ |x|\to \infty.
		\end{equation}		
		Combining this result and \eqref{382}, we obtain
		\begin{equation}\label{387}
			v(x)\simeq C|x|^{2\beta-n} \ \ when \ \ |x|\to \infty.
		\end{equation}

		When $y\in \mathbb{R}^n\setminus B_{2|x|}(0)$, $|x-y|\geq |y|-|x|\geq |y|/2$.
Therefore, from \eqref{386} it follows
		$$
		u(x)\geq \int_{\mathbb{R}^n\setminus B_{2|x|}(0)}
		\frac{|x-y|^{\beta-n}|y|^{\beta}}{v^{p_2}(y)}dy
		\geq C\int_{2|x|}^{\infty}r^{2\beta-(2\beta-n)p_2}\frac{dr}{r}.
		$$
		Since $u$ is an entire solution, we see from the result above that
		$2\beta-(2\beta-n)p_2<0$. Namely,
		\begin{equation}\label{388}
			p_2>\frac{2\beta}{2\beta-n}.
		\end{equation}

(2) Let $|x| \ll 1$.

When $|y| \leq 2|x|$, there holds $|x-y| \leq 3|x|$. Thus,
$$
v_4(x):=\int_{B_{2|x|(0)}}\frac{|x-y|^{\beta-n}|x|^{\beta}}{u^{p_1}(y)}dy
\leq C\left(\max_{|y| \leq 1}u^{-p_1}(y)\right)|x|^{2\beta-n} \int_0^{2|x|}r^n\frac{dr}{r}
\leq C|x|^{2\beta}.
$$
When $y \in B_R(0) \setminus B_{2|x|}(0)$, there holds $|y|/2 \leq |y|-|x|
\leq |x-y| \leq |x|+|y| \leq 3|y|/2$. Therefore,
$$
v_5(x):=\int_{B_R(0) \setminus B_{2|x|}(0)}\frac{|x-y|^{\beta-n}|x|^{\beta}}{u^{p_1}(y)}dy
\leq C\left(\max_{|y| \leq R}u^{-p_1}(y)\right)|x|^\beta \int_{2|x|}^R r^{n+\beta-n}\frac{dr}{r}
\leq C|x|^{\beta},
$$
and
$$
v_5(x) \geq C\left(\min_{|y| \leq R}u^{-p_1}(y)\right)|x|^\beta \int_{2|x|}^R r^{n+\beta-n}\frac{dr}{r}
\geq C|x|^{\beta}.
$$
By \eqref{381} and \eqref{380}, it follows
$$
v_6(x):=\int_{\mathbb{R}^n \setminus B_R(0)}\frac{|x-y|^{\beta-n}|x|^{\beta}}{u^{p_1}(y)}dy
\leq C|x|^\beta \int_R^\infty r^{n+\beta-n-p_1(\beta-n)}\frac{dr}{r}
\leq C|x|^{\beta}.
$$
Combining the estimates of $v_i$ $(i=4,5,6)$ we get
\begin{equation}\label{3872}
			v(x) \simeq C|x|^{\beta} \ \ when \ \ |x| \to 0.
		\end{equation}

Applying \eqref{3872}, we can find $\delta \in (0,1/2)$ such that for $|x| > 2\delta$,
$$
u(x) \geq \int_{B_\delta(0)}\frac{|x-y|^{\beta-n}|y|^{\beta}}{v^{p_2}(y)}dy
\geq c\delta^{\beta-n}\int_0^\delta r^{n+\beta-p_2\beta}\frac{dr}{r}.
$$
Since $u$ is an entire solution, we have $n+\beta-p_2\beta>0$ which implies
$0<p_2<(n+\beta)/\beta$.
Combining this with \eqref{388}, we get
\begin{equation}\label{3882}
\frac{2\beta}{2\beta-n}<p_2<\frac{n+\beta}{\beta}.
\end{equation}

{\it Step 2.} Assume that \eqref{3882} holds. We will prove
		\eqref{380} and $u(x)\simeq |x|^{\beta-n}$ when $|x|\to \infty$.

In fact, by \eqref{387} and \eqref{3872}, we can find $\delta \in (0,1)$ and $R>1$
such that
\begin{equation}\label{nm1}
v(y) \geq c|y|^{2\beta-n} \quad when \quad |y| \geq R,
\end{equation}
and
\begin{equation}\label{nm2}
v(y) \geq c|y|^\beta \quad when \quad |y| \leq \delta.
\end{equation}

Let $|x| \gg 1$.
Using \eqref{nm2} and \eqref{3882} we have
$$
u_1(x):=\int_{B_\delta(0)}\frac{|x-y|^{\beta-n}|y|^{\beta}}{v^{p_2}(y)}dy
\leq C|x|^{\beta-n}\int_0^\delta r^{n+\beta-p_2\beta}\frac{dr}{r} \leq C|x|^{\beta-n}.
$$
Since $v$ is an entire solution, there holds
$$
u_2(x):=\int_{B_R(0) \setminus B_\delta(0)}\frac{|x-y|^{\beta-n}|y|^{\beta}}{v^{p_2}(y)}dy
\leq C\left(\max_{\delta \leq |y| \leq R}v^{-p_2}(y)\right)|x|^{\beta-n}\int_\delta^R r^{n+\beta}\frac{dr}{r}
\leq C|x|^{\beta-n}.
$$
When $|y| \leq 2|x|$, $|x-y| \leq 3|x|$. By \eqref{nm1}, \eqref{3882} and $\beta>n$, there holds
$$
u_3(x):=\int_{B_{2|x|}(0) \setminus B_R(0)}\frac{|x-y|^{\beta-n}|y|^{\beta}}{v^{p_2}(y)}dy
\leq C|x|^{\beta-n}\int_R^{2|x|} r^{n+\beta-p_2(2\beta-n)}\frac{dr}{r}
\leq C|x|^{\beta-n}.
$$
When $|y| \geq 2|x|$, $|x-y| \leq 3|y|/2$. By \eqref{nm1}, \eqref{3882}, there holds
$$
u_4(x):=\int_{\mathbb{R}^n \setminus B_{2|x|}(0)}\frac{|x-y|^{\beta-n}|y|^{\beta}}{v^{p_2}(y)}dy
\leq C\int_{2|x|}^\infty r^{n+\beta-n+\beta-p_2(2\beta-n)}\frac{dr}{r}
\leq C|x|^{2\beta-p_2(2\beta-n)}.
$$
Clearly, \eqref{3882} and $\beta>n$ lead to $\beta-n>2\beta-p_2(2\beta-n)$.
Therefore, estimates of $u_i(x)$ $(i=1,2,3,4)$ imply
$u(x)\leq C|x|^{\beta-n}$ when $|x|\to \infty$. Combining this with \eqref{381},
we see that
\begin{equation}\label{nm3}
u(x)\simeq |x|^{\beta-n}\quad  when \quad |x| \to \infty.
\end{equation}

Next we prove \eqref{380}. If \eqref{380} is not true,
for some large $|x_0|$, from \eqref{j3} and \eqref{nm3} we can deduce
$$
v(x_0) \geq \int_{\mathbb{R}^n \setminus B_{2|x_0|}(0)}
\frac{|x_0-y|^{\beta-n}|x_0|^{\beta}}{u^{p_1}(y)}dy
\geq c|x_0|^\beta \int_{2|x_0|}^\infty r^{n+\beta-n-p_1(\beta-n)}\frac{dr}{r}
=\infty.
$$
It contradicts with $v \in L_{loc}^\infty(\mathbb{R}^n)$.

		The proof of Proposition \ref{prop} is complete.
	\end{proof}

		\begin{prop}\label{prop2} (Integrability)
Let $(u,v)$ be a pair of positive entire solutions of \eqref{j3} with $\beta>n$.
		If $p_1>\beta/(\beta-n)$ or $(n+\beta)/\beta>p_2>2\beta/(2\beta-n)$, then
		$u^{-1}\in L^s(\mathbb{R}^n)$ for all $s>n/(\beta-n)$, and
$v^{-1}\in L^t(\mathbb{R}^n)$ for all $t \in (n/(2\beta-n),n/\beta)$.
	\end{prop}
		
\begin{proof}		
According to Proposition \ref{prop}, there exists $R>0$ sufficiently large, such that
		$$
		u(x) \geq |x|^{\beta-n}, \quad |x|>R.
		$$
		Since $u$ is a positive entire solution, for any $s>0$ we have
		$$
					\int_{\mathbb{R}^n}u^{-s}(x)dx
			=\int_{B_R(0)}u^{-s}(x)dx+\int_{\mathbb{R}^n\setminus B_R(0)}u^{-s}(x)dx
			\leq C+C\int_{R}^{\infty}r^{n-s(\beta-n)}\frac{dr}{r}.
		$$
		Therefore, $u^{-1}\in L^s(\mathbb{R}^n)$ as long as $s>n/(\beta-n)$.
		
		Similarly, Proposition \ref{prop} shows that there exist large $R>1$ and small $\delta
\in (0,1)$ such that $v(x) \geq |x|^{2\beta-n}$ when $|x|>R$, and $v(x) \geq |x|^{\beta}$ when $|x|<\delta$.
Therefore, for any $t>0$, we have
$$
		\begin{aligned}
			\int_{\mathbb{R}^n}v^{-t}(x)dx
			=&\int_{B_\delta(0)}v^{-t}(x)dx+\int_{B_R(0) \setminus B_\delta(0)}v^{-t}(x)dx
+\int_{\mathbb{R}^n\setminus B_R(0)}v^{-t}(x)dx\\
			\leq& C\int_{0}^{\delta}r^{n-t\beta}\frac{dr}{r}+C
+C\int_{R}^{\infty}r^{n-t(2\beta-n)}\frac{dr}{r}.
		\end{aligned}
		$$
		Therefore, $v^{-1}\in L^t(\mathbb{R}^n)$ as long as $n/(2\beta-n)<t<n/\beta$.
\end{proof}
		
		\begin{prop}\label{prop3} (Necessary condition)
Let $(u,v)$ be a pair of positive differentiable entire solutions of \eqref{j3} with $\beta>n$.
		If either \eqref{380} or \eqref{3882} holds, \eqref{105} must be true. 		
	\end{prop}
		
\begin{proof}		We use the Pohozaev identity in integral form to prove \eqref{105}.
		
		{\it Step 1.}
Proposition \ref{prop} shows that \eqref{380} is equivalent to \eqref{3882}. If one holds,
		the other is still true. From \eqref{380} and \eqref{3882}, it follows that
		$$
		p_1-1>\frac{n}{\beta-n},\quad p_2-1 \in \left(\frac{n}{2\beta-n},\frac{n}{\beta}\right).
		$$
		Therefore, Proposition \ref{prop2} implies
		$$
		(u^{-1}, v^{-1})\in L^{p_1-1}(\mathbb{R}^n)\times L^{p_2-1}(\mathbb{R}^n).
		$$
		According to the definition of the improper integral, we get
		$$
		0=\lim_{\mu\to \infty}\int_{B_{2\mu}(0)\setminus B_{\mu}(0)}u^{1-p_1}(x)dx
		=\lim_{\mu\to \infty}\int_{\mu}^{2\mu}[r\int_{\partial B_r(0)}u^{1-p_1}(x)ds]\frac{dr}{r},
		$$
		and hence
		$$
		\inf_{[\mu,2\mu]}r\int_{\partial B_r(0)}u^{1-p_1}(x)ds\to 0
		$$
		when $\mu \to \infty$. This implies that there exists $R_j\to \infty$ such that
		\begin{equation}\label{800}
			R_j\int_{\partial B_{R_j}(0)}u^{1-p_1}(x)ds\to 0.
		\end{equation}

		Next, integrating by parts, we get
		$$
		\int_{B_R(0)}x\cdot\nabla u^{1-p_1}dx=R\int_{\partial B_R(0)} u^{1-p_1}(x)ds-n\int_{B_R(0)}u^{1-p_1}(x)dx.
		$$
		Letting $R=R_j\to \infty$ and using \eqref{800}, we obtain $
		(x\cdot \nabla u^{1-p_1})\in L^1(\mathbb{R}^n)$, and
		\begin{equation}\label{801}
			\int_{\mathbb{R}^n}x\cdot\nabla u^{1-p_1}dx=-n\int_{\mathbb{R}^n}u^{1-p_1}(x)dx.
		\end{equation}
		Therefore,
		\begin{equation}\label{802}
			\int_{\mathbb{R}^n}\frac{x\cdot\nabla u(x)}{u^{p_1}(x)}dx
			=\frac{1}{1-p_1}\int_{\mathbb{R}^n}x\cdot \nabla u^{1-p_1}(x)dx
			=\frac{n}{p_1-1}\int_{\mathbb{R}^n}u^{1-p_1}(x)dx.
		\end{equation}
		
		Similar to the derivation of \eqref{801}, we also have $(x\cdot \nabla v^{1-p_2})\in L^1(\mathbb{R}^n)$ and
		\begin{equation}\label{803}
			\int_{\mathbb{R}^n}x\cdot\nabla v^{1-p_2}dx=-n\int_{\mathbb{R}^n}v^{1-p_2}(x)dx.
		\end{equation}

{\it Step 2. We claim
\begin{equation}\label{125}
\left|\int_{\mathbb{R}^n}|x-z|^{\beta-n}|z|^{\beta}(z\cdot \nabla v^{-p_2}(z))dz\right|<\infty
\end{equation}
at each $x \in \mathbb{R}^n$.}

Integrating by parts, we have
$$\begin{aligned}
&\int_{B_R(0)}|x-z|^{\beta-n}|z|^{\beta}(z\cdot \nabla v^{-p_2}(z))dz\\
=&\int_{\partial B_R(0)}|x-z|^{\beta-n}|z|^{\beta+1} v^{-p_2}(z)ds
-(n+\beta)\int_{B_R(0)}|x-z|^{\beta-n}|z|^{\beta} v^{-p_2}(z)dz\\
&+(\beta-n)\int_{B_R(0)}|x-z|^{\beta-n-2}|z|^{\beta}[z\cdot(x-z)]  v^{-p_2}(z)dz\\
:=&Y_1-(n+\beta)Y_2+(\beta-n)Y_3.	
\end{aligned}
$$		

According to Proposition \ref{prop}, by \eqref{388} we get
$$
0 \leq Y_1 \leq CR^{2\beta-p_2(2\beta-n)} \to 0 \quad when \quad R \to \infty.
$$
Clearly,
$$
\lim_{R \to \infty} Y_2 = u(x).
$$
To obtain the limit of $Y_3$, write
$$
Y(x;\Omega):=\int_{\Omega}|x-z|^{\beta-n-2}|z|^{\beta}[z\cdot(x-z)]  v^{-p_2}(z)dz.
$$
We observe that the defects of the improper integral
$Y(x;\mathbb{R}^n)$ may happen at $x$ and $\infty$.

When $x=0$, noting the asymptotic rates
in Proposition \ref{prop}, we can find $\delta \in (0,1)$ and $\rho>1$ such that
$$\begin{aligned}
|Y(0;\mathbb{R}^n)| &\leq |Y(0;B_\delta(0))|+|Y(0;B_\rho(0) \setminus B_\delta(0))|
+|Y(0;\mathbb{R}^n \setminus B_\rho(0))|\\
&\leq C\int_0^\delta r^{2\beta-p_2\beta}\frac{dr}{r}+C
+C\int_\rho^\infty r^{2\beta-p_2(2\beta-n)}\frac{dr}{r}<\infty
\quad (by \ \eqref{3882}).
\end{aligned}
$$
When $x \neq 0$, using Proposition \ref{prop},
we can find $\rho>\max\{1,2|x|\}$ and $\delta \in (0,\min\{1,|x|/2\})$
such that
$$\begin{aligned}
|Y(x;\mathbb{R}^n)| &\leq |Y(x;B_\delta(x))|+|Y(x;B_\rho(0) \setminus B_\delta(x))|
+|Y(x;\mathbb{R}^n \setminus B_\rho(0))|\\
&\leq C\left(\max_{|z-x| \leq \delta}|z|^{\beta+1}v^{-p_2}(z)\right)
\int_0^\delta r^{\beta-1}\frac{dr}{r}+C
+C\int_\rho^\infty r^{2\beta-p_2(2\beta-n)}\frac{dr}{r}\\
&<\infty \quad (by \ \ \eqref{388}).
\end{aligned}
$$
Therefore,
$Y_3 \to Y(x;\mathbb{R}^n)$ when $R \to \infty$.
Combining the convergence results of $Y_1,Y_2$ and $Y_3$, we can see \eqref{125}.

		{\it Step 3.}
		For $\mu>0$,
		$$
		u(\mu x)=\int_{\mathbb{R}^n}\frac{|\mu x-y|^{\beta-n}|y|^{\beta}}{v^{p_2}(y)}dy
		=\mu^{2\beta}\int_{\mathbb{R}^n}\frac{|x-z|^{\beta-n}|z|^{\beta}}{v^{p_2}(\mu z)}dz.
		$$
		Differentiating both sides with respect to $\mu$ and then letting $\mu=1$, we get			
\begin{equation}\label{130}
x\cdot \nabla u(x)=\left[\frac{du(\mu x)}{d\mu}\right]_{\mu=1}
		=2\beta u(x)+\int_{\mathbb{R}^n}|x-z|^{\beta-n}|z|^{\beta}(z\cdot\nabla v^{-p_2}(z))dz.
\end{equation}		
By \eqref{125} we know that \eqref{130} makes sense.
		Multiplying \eqref{130} by $u^{-p_1}(x)$ and integrating on $B_R(0)$,
		we obtain
		$$\begin{aligned}
		&\int_{B_R(0)}\frac{x\cdot \nabla u(x)}{u^{p_1}(x)}dx
		-2\beta \int_{B_R(0)}u^{1-p_1}(x)dx\\
&=\int_{B_R(0)}u^{-p_1}(x)
\int_{\mathbb{R}^n}|x-z|^{\beta-n}|z|^{\beta}(z\cdot \nabla v^{-p_2}(z))dzdx.
\end{aligned}
		$$
Letting $R \to \infty$, from \eqref{802} we know that the limit of the right hand side exists, and hence
\begin{equation}\label{804}
\left(\frac{n}{p_1-1}-2\beta\right)\int_{\mathbb{R}^n}u^{1-p_1}(x)dx=\int_{\mathbb{R}^n}u^{-p_1}(x)
\int_{\mathbb{R}^n}|x-z|^{\beta-n}|z|^{\beta}(z\cdot \nabla v^{-p_2}(z))dzdx.
\end{equation}

On the other hand, applying \eqref{j3} and the Fubini theorem, we have
		\begin{equation*}
			\begin{aligned}
				\frac{p_2}{p_2-1}\int_{\mathbb{R}^n}z\cdot \nabla v^{1-p_2}(z)dz
				=&\int_{\mathbb{R}^n}(z\cdot \nabla v^{-p_2}(z))v(z)dz\\
				=&\int_{\mathbb{R}^n}(z\cdot \nabla v^{-p_2}(z))\int_{\mathbb{R}^n}\frac{|x-z|^{\beta-n}|z|^{\beta}}{u^{p_1}(x)}dxdz\\
				=&\int_{\mathbb{R}^n}u^{-p_1}(x)\int_{\mathbb{R}^n}|x-z|^{\beta-n}|z|^{\beta}(z\cdot \nabla v^{-p_2}(z))dzdx.
			\end{aligned}
		\end{equation*}
		Inserting \eqref{803} and \eqref{804} into the above result yields
		\begin{equation}\label{805}
			\left(\frac{n}{p_1-1}-2\beta\right)\int_{\mathbb{R}^n}u^{1-p_1}(x)dx
			=-\frac{np_2}{p_2-1}\int_{\mathbb{R}^n}v^{1-p_2}(x)dx.
		\end{equation}
		
		Applying \eqref{j3} and the Fubini theorem, we have
		$$
		\begin{aligned}
			\int_{\mathbb{R}^n}u^{1-p_1}(x)dx
			=&\int_{\mathbb{R}^n}u^{-p_1}(x)\int_{\mathbb{R}^n}\frac{|x-y|^{\beta-n}|y|^{\beta}}{v^{p_2}(y)}dydx\\
			=&\int_{\mathbb{R}^n}v^{-p_2}(y)\int_{\mathbb{R}^n}\frac{|x-y|^{\beta-n}|y|^{\beta}}{u^{p_1}(x)}dxdy
			=\int_{\mathbb{R}^n}v^{1-p_2}(x)dx.
		\end{aligned}
		$$
		Substituting this result into \eqref{805}, we obtain
		\eqref{105}.
	\end{proof}

\section{Proof of Theorem \ref{th10}}

		\textbf{Proof of (i).}  Let $n \geq 1$, $\gamma \in (0,n)$, $\min\{q_1,q_2\}>0$.
		
		Assuming that $(u,v) \in L_{loc}^{\infty}(\mathbb{R}^n \setminus \{0\}) \times
L_{loc}^{\infty}(\mathbb{R}^n \setminus \{0\})$ is a pair of positive super solutions of \eqref{j8},
we will deduce a contradiction.

{\it Case 1:} $0<q_1q_2<1$.
		
		For $R>0$ and $|x|>4R$, there holds
		\begin{equation}\label{j13}
			u(x) \geq c\int_{B_{2R}(0)\setminus B_R(0)}\frac{v^{q_2}(y)}{|x-y|^{n-\gamma}|y|^{\gamma}}dy
			\geq \frac{c}{|x|^{n-\gamma}}.
		\end{equation}
		Similarly, we also have
		$$
		v(x)\geq \frac{c}{|x|^n}, \quad \forall \ |x|>4R.
		$$
		Therefore, from the first equation of \eqref{j8} we obtain, for $|x|>8R$,
		\begin{equation}\label{j14}
				u(x)\geq c\int_{B_{|x|/2}(x)}\frac{dy}{|x-y|^{n-\gamma}|y|^{\gamma}|y|^{q_2n}}
				\geq \frac{c}{|x|^{\gamma+q_2n}}
				\int_{B_{|x|/2}(x)}\frac{dy}{|x-y|^{n-\gamma}}
				=\frac{c}{|x|^{q_2n}}.
		\end{equation}
		
		Set $a_0:=\min\{n-\gamma, q_2n\}$. By \eqref{j13} and \eqref{j14}, there holds
		$$
		u(x)\geq \frac{c}{|x|^{a_0}}, \quad for \ large \quad |x|.
		$$
		Inserting this estimate into the second equation of \eqref{j8}, we have
		$$
		v(x)\geq c\int_{B_{|x|/2}(x)}\frac{dy}{|x-y|^{n-\gamma}|x|^{\gamma}|y|^{q_1a_0}}
		\geq \frac{c}{|x|^{q_1a_0}}:=\frac{c}{|x|^{b_1}},
		$$
		which implies
		$$
		u(x)\geq c\int_{B_{|x|/2}(x)}\frac{dy}{|x-y|^{n-\gamma}|y|^{\gamma}|y|^{q_2b_1}}
		\geq \frac{c}{|x|^{q_2b_1}}:=\frac{c}{|x|^{a_1}}.
		$$
		By an induction argument, we obtain that
		$$
u(x)\geq \frac{c}{|x|^{a_k}}, \quad v(x)\geq \frac{c}{|x|^{b_k}}, \quad for \; large \quad |x|.
		$$
		Here $a_0=\min\{n-\gamma, q_2n\}$, $a_k=q_2b_{k}$ and $b_k=q_1a_{k-1}$ $(k=1,2,\cdots)$.
		
		When $0<q_1q_2<1$, we deduce that
		$$
		a_{j}=q_1q_2a_{j-1}=(q_1q_2)^{j}a_0 \to 0 \quad (j \to \infty).
		$$
		Therefore, $u(x) \geq c|x|^{-\epsilon}$ when $|x|$ is suitably large,
where $\epsilon \in (0,\gamma/q_1)$ is sufficiently small.
		
		When $|y|>2|x|$, there holds $|x-y|\leq 3|y|/2$. Therefore,
		$$
		v(x) \geq c\int_{\mathbb{R}^n \setminus B_{2|x|}(0)}\frac{dy}{|x-y|^{n-\gamma}|x|^{\gamma}|y|^{q_1\epsilon}}
        \geq c|x|^{-\gamma} \int_{2|x|}^{\infty}r^{\gamma-q_1\epsilon}\frac{dr}{r}
		= \infty.
		$$
		It is impossible by virtue of $v \in L_{loc}^{\infty}(\mathbb{R}^n \setminus \{0\})$.

{\it Case 2:} $0 < q_1 \leq \gamma/(n-\gamma)$.

In view of \eqref{j13}, we obtain that for large $|x|$,
$$
v(x) \geq \frac{c}{|x|^\gamma} \int_{\mathbb{R}^n \setminus B_{2|x|}(0)}
\frac{dy}{|x-y|^{n-\gamma}|y|^{q_1(n-\gamma)}}
\geq \frac{c}{|x|^\gamma} \int_{2|x|}^\infty
r^{n-(n-\gamma)-q_1(n-\gamma)}
\frac{dr}{r}=\infty.
$$
This contradiction implies the nonexistence result.

{\it Case 3:} $q_2 \geq \gamma/(n-\gamma)$.

Let $|x| \ll 1$. When $|y| \geq 2|x|$, there holds $|x-y| \leq 3|y|/2$. Thus,
$$
v(x) \geq c\int_{B_1(0) \setminus B_{2|x|}(0)} \frac{u^{q_1}(y) dy}{|x-y|^{n-\gamma}|x|^{\gamma}}
\geq \frac{c}{|x|^\gamma}\left(\min_{|y| \leq 1}u^{q_1}(y)\right) \int_{2|x|}^1 r^{n-(n-\gamma)} \frac{dr}{r}
\geq \frac{c}{|x|^\gamma}.
$$
Therefore, for $x_0 \neq 0$, we can find $\delta \in (0,|x_0|/4)$ such that
$$
u(x_0) \geq c\int_{B_\delta(0)}\frac{v^{q_2}(y) dy}{|x_0-y|^{n-\gamma}|y|^{\gamma}}
\geq \frac{c}{|x_0|^{n-\gamma}} \int_0^{\delta} r^{n-\gamma-q_2\gamma} \frac{dr}{r}
=\infty.
$$
It is impossible.

\vskip 5mm
		
\textbf{Proof of (ii).}		
We will prove that \eqref{j8} has positive radial single super solutions (and lower solutions).

Clearly, $q_1q_2=1$ and $\min\{q_1,q_2\}>0$ show that
$q_1 > \gamma/(n-\gamma)$ is equivalent to $0<q_2<(n-\gamma)/\gamma$.
Now, $q_1 > \gamma/(n-\gamma)$ and $0<q_2<(n-\gamma)/\gamma$ respectively lead to
$$
I_1:=\left(\frac{\gamma}{q_1},\frac{n}{q_1}\right) \cap (0,n-\gamma) \neq \emptyset
\quad and \quad
I_2:=\left(0,\frac{n-\gamma}{q_2}\right) \cap (\gamma,n) \neq \emptyset.
$$
Therefore, we can choose 		
\begin{equation}\label{j16}
			\sigma_1 \in I_1
		\end{equation}
		and
		\begin{equation}\label{j23}
			\sigma_2 \in I_2
		\end{equation}		
satisfying
\begin{equation}\label{sq}
\sigma_1=q_2\sigma_2, \quad \sigma_2=q_1\sigma_1
\end{equation}
because of $q_1q_2=1$. Set
		\begin{equation}\label{j22}
			u(x)=|x|^{-\sigma_1},\quad
			v(x)=|x|^{-\sigma_2}.
		\end{equation}
		
We call $f(x)$ a double bounded function,
if there exists $C>0$ such that $C^{-1} \leq f(x) \leq C$ for all $x \in \mathbb{R}^n$.
	
		\textit{Step 1.} $|x|^{\gamma}v(x)[|x|^{n-\gamma}\ast u^{q_1}(x)]^{-1}$ is double bounded.

		We will consider three cases: $|x|\gg 1$, $|x|\ll 1$, and $|x|$ is double bounded.
		
		\textit{Case 1.1:} $|x|\gg 1$.
		
		Insert \eqref{j22} with \eqref{j16} into the right hand side of the second equation of \eqref{j8}.
For suitably small $\delta \in (0,1)$, we have
		\begin{equation}\label{j15}
			\begin{aligned}
				\int_{\mathbb{R}^n}&\frac{u^{q_1}(y)dy}{|x-y|^{n-\gamma}|x|^{\gamma}}
				= \int_{B_{\delta}(0)}\frac{u^{q_1}(y)dy}{|x-y|^{n-\gamma}|x|^{\gamma}}
				+\int_{B_{|x|/2}(x)}\frac{u^{q_1}(y)dy}{|x-y|^{n-\gamma}|x|^{\gamma}}\\
				&+ \int_{(B_{2|x|}(0)\setminus B_{\delta}(0))\setminus B_{|x|/2}(x)}
\frac{u^{q_1}(y)dy}{|x-y|^{n-\gamma}|x|^{\gamma}}
				+\int_{\mathbb{R}^n\setminus B_{2|x|}(0)}\frac{u^{q_1}(y)dy}{|x-y|^{n-\gamma}|x|^{\gamma}}
				:=  \sum_{i=1}^4E_i.
			\end{aligned}
		\end{equation}
		
		When $y\in B_{\delta}(0)$, there holds $|x|/2\leq |x-y| \leq 2|x|$.
By \eqref{j16}, there exists a positive constant $C>1$ such that
		\begin{equation}\label{j17}
		E_1
		\leq \frac{C}{|x|^n}\int_{B_{\delta}(0)}\frac{dy}{|y|^{q_1\sigma_1}} \leq C|x|^{-n}.
		\end{equation}

		When $y\in B_{|x|/2}(x)$, there holds $|x|/2\leq |y|\leq 3|x|/2$. Therefore, we can find
$C>1$ such that
		\begin{equation}\label{j18}
			E_2 \leq \frac{C}{|x|^{\gamma+q_1\sigma_1}}\int_{B_{|x|/2}(x)}\frac{dy}{|x-y|^{n-\gamma}}
			\leq \frac{C}{|x|^{\gamma+q_1\sigma_1}}\int_0^{|x|/2}r^{\gamma-1}dr
			\leq \frac{C}{|x|^{q_1\sigma_1}}.
		\end{equation}
		
		When $y\in (B_{2|x|}(0)\setminus B_{\delta}(0))\setminus B_{|x|/2}(x)$,
		there holds $|x|/2\leq |x-y|\leq 3|x|$. Thus, there exists a positive constant $C>1$ such that
		\begin{equation}\label{j19}
			E_3 \leq \frac{C}{|x|^n}\int_{\delta}^{2|x|}r^{n-q_1\sigma_1}\frac{dr}{r}
			\leq \frac{C}{|x|^{q_1\sigma_1}}.
		\end{equation}
		
		When $y\in \mathbb{R}^n\setminus B_{2|x|}(0)$, there holds $|y|/2\leq |x-y|\leq 3|y|/2$.
		In view of \eqref{j16}, we can find $C>1$ such that
		\begin{equation}\label{j20}
			\frac{1}{C|x|^{q_1\sigma_1}} \leq
\frac{1}{C|x|^{\gamma}}\int_{2|x|}^{\infty}r^{\gamma-q_1\sigma_1}\frac{dr}{r} \leq E_4 \leq \frac{C}{|x|^{\gamma}}\int_{2|x|}^{\infty}r^{\gamma-q_1\sigma_1}\frac{dr}{r} \leq \frac{C}{|x|^{q_1\sigma_1}}.
		\end{equation}

		Obviously, \eqref{j16} implies $q_1\sigma_1<n$. Inserting \eqref{j17}-\eqref{j20} into \eqref{j15},
and noting $E_i \geq 0$ for $i=1,2,3$, we can find a positive constant $C>1$ such that
\begin{equation}\label{j21}		
		\frac{1}{C|x|^{q_1\sigma_1}}
		\leq \int_{\mathbb{R}^n}\frac{u^{q_1}(y)dy}{|x-y|^{n-\gamma}|x|^{\gamma}}
		\leq \frac{C}{|x|^{q_1\sigma_1}}, \quad (|x|\gg 1).
\end{equation}

		\textit{Case 1.2:} $|x|\ll 1$.
		
		Similar to \eqref{j15}, for suitably small $\delta \in (0,1)$, we have
		$$
		\begin{aligned}
			\int_{\mathbb{R}^n}&\frac{u^{q_1}(y)dy}{|x-y|^{n-\gamma}|x|^{\gamma}}
			= \int_{B_{|x|/2}(x)}\frac{u^{q_1}(y)dy}{|x-y|^{n-\gamma}|x|^{\gamma}}
			+\int_{B_{2|x|}(0)\setminus B_{|x|/2}(x)}
			\frac{u^{q_1}(y)dy}{|x-y|^{n-\gamma}|x|^{\gamma}}\\
			&+ \int_{B_{\delta}(0)\setminus B_{2|x|}(0)}
			\frac{u^{q_1}(y)dy}{|x-y|^{n-\gamma}|x|^{\gamma}}
			+\int_{\mathbb{R}^n\setminus B_{\delta}(0)}
			\frac{u^{q_1}(y)dy}{|x-y|^{n-\gamma}|x|^{\gamma}}
			:=  \sum_{i=1}^4F_i.
		\end{aligned}
		$$
		
		When $y\in B_{|x|/2}(x)$, there holds $|x|/2\leq |y| \leq 3|x|/2$.
		Therefore, it follows that
		$$
		\frac{1}{C|x|^{\gamma+q_1\sigma_1}}\int_{B_{|x|/2}(x)}\frac{dy}{|x-y|^{n-\gamma}}
\leq F_1 \leq \frac{C}{|x|^{\gamma+q_1\sigma_1}}\int_{B_{|x|/2}(x)}\frac{dy}{|x-y|^{n-\gamma}}.
$$
Clearly,
$$
\frac{1}{C} \int_0^{|x|/2} r^\gamma \frac{dr}{r} \leq \int_{B_{|x|/2}(x)}\frac{dy}{|x-y|^{n-\gamma}}
 \leq C\int_0^{|x|/2} r^\gamma \frac{dr}{r}.
$$
Therefore,
$$
\frac{1}{C|x|^{q_1\sigma_1}}	\leq F_1	
\leq \frac{C}{|x|^{q_1\sigma_1}}.
		$$
		
		When $y\in B_{2|x|}(0)\setminus B_{|x|/2}(x)$, there holds $|x|/2\leq |x-y| \leq 3|x|$.
		In view of \eqref{j16}, we have
		$$
		F_2 \leq \frac{C}{|x|^n}\int_{0}^{2|x|}r^{n-q_1\sigma_1}\frac{dr}{r}
		\leq \frac{C}{|x|^{q_1\sigma_1}}.
		$$
		
		When $y\in B_{\delta}(0)\setminus B_{2|x|}(0)$, there holds $|y|/2\leq |x-y| \leq 3|y|/2$.
		Then, it follows
		$$
		F_3 \leq \frac{C}{|x|^{\gamma}}\int_{2|x|}^{\delta}r^{\gamma-q_1\sigma_1}\frac{dr}{r}
		\leq \frac{C}{|x|^{q_1\sigma_1}}.
		$$
		
		When $y\in \mathbb{R}^n\setminus B_{\delta}(0)$, there holds $|y|/2\leq |x-y| \leq 2|y|$.
		In view of $\gamma<q_1\sigma_1$ which is implied by \eqref{j16}, we have
		$$
		F_4 \leq \frac{C}{|x|^{\gamma}}\int_{\delta}^{\infty}r^{\gamma-q_1\sigma_1}\frac{dr}{r}
		\leq \frac{C}{|x|^{\gamma}}.
		$$

		Clearly, $F_i \geq 0$ for $i=2,3,4$. Combining the estimates of $F_i$ $(i=1,2,3,4)$, we know
that
		\begin{equation}\label{j24}
			\frac{1}{C|x|^{q_1\sigma_1}} \leq \int_{\mathbb{R}^n}\frac{u^{q_1}(y)dy}{|x-y|^{n-\gamma}|x|^{\gamma}}
			\leq \frac{C}{|x|^{q_1\sigma_1}}, \quad (|x|\ll 1).
		\end{equation}

		\textit{Case 1.3:} $|x|$ is double bounded.
		
		In view of \eqref{j16}, we know that $0 \leq E_i \leq C$ $(i=1,2,3)$, and $E_4$ is double bounded.
Therefore, when $|x|$ is double bounded,
$$
\frac{1}{C|x|^{q_1\sigma_1}} \leq \int_{\mathbb{R}^n}\frac{u^{q_1}(y)dy}{|x-y|^{n-\gamma}|x|^{\gamma}}
			\leq \frac{C}{|x|^{q_1\sigma_1}}.
$$

		Combining this result with \eqref{j21} and \eqref{j24},
and using \eqref{j22} and \eqref{sq}, we can find $C>1$ such that
		\begin{equation}\label{j25}
			C^{-1}v(x) \leq \int_{\mathbb{R}^n}\frac{u^{q_1}(y)dy}{|x-y|^{n-\gamma}|x|^{\gamma}}
\leq Cv(x)
		\end{equation}
holds for all $x\in \mathbb{R}^n\setminus \{0\}$.	

\vskip 3mm
	
		\textit{Step 2.} $u(x)[|x|^{n-\alpha/q'}\ast (|x|^{\alpha/q'}u^{q_1}(x))]^{-1}$ is double bounded.
		
		\textit{Case 2.1:} $|x|\gg 1$.
		Insert \eqref{j22} into the right hand
side of the first equation of \eqref{j8}. For suitably small $\delta\in (0,1)$, we have
		$$
		\begin{aligned}
\int_{\mathbb{R}^n}&\frac{v^{q_2}(y)dy}{|x-y|^{n-\gamma}|y|^{\gamma}}
			=  \int_{B_{\delta}(0)}\frac{v^{q_2}(y)dy}{|x-y|^{n-\gamma}|y|^{\gamma}}
			+\int_{B_{|x|/2}(x)}\frac{v^{q_2}(y)dy}{|x-y|^{n-\gamma}|y|^{\gamma}}\\
			& + \int_{(B_{2|x|}(0)\setminus B_{\delta}(0))\setminus B_{|x|/2}(x)}\frac{v^{q_2}(y)dy}{|x-y|^{n-\gamma}|y|^{\gamma}}
			+\int_{\mathbb{R}^n\setminus B_{2|x|}(0)}\frac{v^{q_2}(y)dy}{|x-y|^{n-\gamma}|y|^{\gamma}}
			:=  \sum_{i=1}^4G_i.
		\end{aligned}
		$$
		By the same argument of Case 1.1, using \eqref{j23} we get
		$$
		G_1\leq \frac{C}{|x|^{n-\gamma}}\int_{0}^{\delta}r^{n-\gamma-q_2\sigma_2}\frac{dr}{r}
		\leq \frac{C}{|x|^{n-\gamma}},
		\quad
		G_2 \leq \frac{C}{|x|^{\gamma+q_2\sigma_2}}\int_{0}^{\delta}r^{\gamma}\frac{dr}{r}
		\leq \frac{C}{|x|^{q_2\sigma_2}},
		$$
		$$
		G_3
		\leq \frac{C}{|x|^{n-\gamma}}\int_{\delta}^{2|x|}r^{n-\gamma-q_2\sigma_2}\frac{dr}{r}
		\leq \frac{C}{|x|^{q_2\sigma_2}},
		$$
		and
		$$
		\frac{1}{C|x|^{q_2\sigma_2}} \leq \frac{1}{C}\int_{2|x|}^\infty r^{-q_2\sigma_2}\frac{dr}{r}
\leq G_4 \leq C\int_{2|x|}^\infty r^{-q_2\sigma_2}\frac{dr}{r}
\leq \frac{C}{|x|^{q_2\sigma_2}}.
		$$
		Clearly, $G_i \geq 0$ $(i=1,2,3)$. In view of $q_2\sigma_2 \leq n-\gamma$ (implied by \eqref{j23}),
we obtain from the estimates of $G_i$ above that
\begin{equation}\label{j27}
\frac{1}{C|x|^{q_2\sigma_2}} \leq \sum_{i=1}^4 G_i \leq \frac{C}{|x|^{q_2\sigma_2}}, \quad (|x|\gg 1).
\end{equation}

		\textit{Case 2.2:} $|x|\ll 1$.
		
		Now, we have
		$$
		\begin{aligned}
			\int_{\mathbb{R}^n}\frac{v^{q_2}(y)dy}{|x-y|^{n-\gamma}|y|^{\gamma}}
			= & \int_{B_{|x|/2}(x)}\frac{v^{q_2}(y)dy}{|x-y|^{n-\gamma}|y|^{\gamma}}
			+\int_{B_{2|x|}(0)\setminus B_{|x|/2}(x)}
			\frac{v^{q_2}(y)dy}{|x-y|^{n-\gamma}|y|^{\gamma}}\\
			&+ \int_{\mathbb{R}^n\setminus  B_{2|x|}(0)}
			\frac{v^{q_2}(y)dy}{|x-y|^{n-\gamma}|y|^{\gamma}}
=  \sum_{i=1}^3 H_i.
		\end{aligned}
		$$
		By the same argument of Case 1.2, using \eqref{j23} we get
		$$
		\frac{1}{C|x|^{q_2\sigma_2}} \leq \frac{1}{C|x|^{\gamma+q_2\sigma_2}}\int_{0}^{|x|/2}r^{\gamma}\frac{dr}{r}
\leq H_1 \leq \frac{C}{|x|^{\gamma+q_2\sigma_2}}\int_{0}^{|x|/2}r^{\gamma}\frac{dr}{r}
		\leq \frac{C}{|x|^{q_2\sigma_2}},
		$$
		$$
		H_2 \leq \frac{C}{|x|^{n-\gamma}}\int_{0}^{2|x|}r^{n-\gamma-q_2\sigma_2}\frac{dr}{r}
		\leq \frac{C}{|x|^{q_2\sigma_2}},
		\quad
		H_3 \leq C\int_{2|x|}^{\infty}r^{-q_2\sigma_2}\frac{dr}{r}
		\leq \frac{C}{|x|^{q_2\sigma_2}}.
		$$

		Combining the estimates of $H_i$ $(i=1,2,3)$, we have
		\begin{equation}\label{j28}
			\frac{1}{C|x|^{q_2\sigma_2}} \leq \int_{\mathbb{R}^n}\frac{v^{q_2}(y)dy}{|x-y|^{n-\gamma}|y|^{\gamma}}
			\leq \frac{C}{|x|^{q_2\sigma_2}}, \quad (|x|\ll 1).
		\end{equation}
		
		\textit{Case 2.3:} $|x|$ is double bounded.
		
		Now, by the analogous argument of Case 1.3, we easily get
		\begin{equation*}
			\frac{1}{C|x|^{q_2\sigma_2}} \leq \int_{\mathbb{R}^n}\frac{v^{q_2}(y)dy}{|x-y|^{n-\gamma}|y|^{\gamma}}
\leq \frac{C}{|x|^{q_2\sigma_2}}.
		\end{equation*}
		
		Combining this result with \eqref{j27} and \eqref{j28}, and using \eqref{j22} and \eqref{sq},
we can see
$$
			C^{-1}u(x) \leq \int_{\mathbb{R}^n}\frac{v^{q_2}(y)dy}{|x-y|^{n-\gamma}|y|^{\gamma}} \leq Cu(x)
$$
for all $x\in \mathbb{R}^n\setminus \{0\}.$
		
		From this result and \eqref{j25}, we find a
super solution and a lower solution as in \eqref{j22}.

\section{Proof of Theorem \ref{th1.6}}

Write $U(r)=U(|x|)=u(x)$ and $V(r)=V(|x|)=v(x)$. Assume that $(U,V)$ is
a pair of decreasing super solutions. We will deduce a contradiction.

Let $|x| \ll 1$. For $y \in B_{|x|}(0)$, we have $|x-y| \leq 2|x|$. Therefore,
$$
U(|x|) \geq c\int_{B_{|x|}(0)} \frac{V^{q_2}(|y|)dy}{|x-y|^{n-\gamma}|y|^\gamma}
\geq c V^{q_2}(|x|) |x|^{\gamma-n} \int_0^{|x|} r^{n-\gamma}\frac{dr}{r} \geq c_1 V^{q_2}(|x|),
$$
Therefore,
$$\begin{aligned}
V(|x|) \geq & c\int_{B_{|x|}(0)} \frac{U^{q_1}(|y|)dy}{|x-y|^{n-\gamma}|x|^\gamma}
\geq c U^{q_1}(|x|) |x|^{-(n-\gamma)-\gamma} \int_{B_{|x|}(0)}dy\\
\geq & c_2 U^{q_1}(|x|) \geq c_1^{q_1}c_2 V^{q_1q_2}(|x|).
\end{aligned}
$$
Here positive constants $c_1$ and $c_2$ are independent of $|x|$.
In view of $q_1q_2>1$, we have
\begin{equation}\label{rad}
V(|x|) \leq (c_1^{-q_1}c_2^{-1})^{\frac{1}{q_1q_2-1}}, \quad |x| \ll 1.
\end{equation}

On the other hand, for $|x| \leq 1/2$,
\begin{equation}\label{Ue}
U(|x|) \geq c V^{q_2}(2)
\int_{B_2(0) \setminus B_1(0)} \frac{dy}{|x-y|^{n-\gamma}|y|^\gamma}
\geq c.
\end{equation}
Let $|x| \ll 1$. When $|y| \geq 2|x|$, $|x-y| \leq 3|y|/2$. Therefore,
by \eqref{Ue} we get
$$
V(|x|) \geq \frac{c}{|x|^\gamma} \int_{B_{1/2}(0) \setminus B_{2|x|}(0)}
\frac{dy}{|x-y|^{n-\gamma}}
\geq \frac{c}{|x|^\gamma},
$$
where $c>0$ is independent of $|x|$. This contradicts with \eqref{rad}.
	
\section{Proof of Theorem \ref{th1.7}}

(i) Suppose $(u,v) \in [L^{q_1+1}(\mathbb{R}^n)
	\cap L^{s}(\mathbb{R}^n)] \times L^{q_2+1}(\mathbb{R}^n)$
	for some $s>nq_1/(\gamma-1)$. We verify $q_1q_2=1$.
	
	{\it Step 1. We claim that the improper integral
	\begin{equation}\label{jia}
		\lim_{d \to 0}\int_{\mathbb{R}^{n} \setminus B_d(0)}
		\frac{z\cdot\nabla u^{q_1}(z)}{|x-z|^{n-\gamma}}dz<\infty
	\end{equation}
	at each $x \in \mathbb{R}^n \setminus \{0\}$.}
	
	In fact, by the same derivation of \eqref{800},
using $u \in L^{q_1+1}(\mathbb{R}^n)$, we can find $R=R_j \to \infty$
	and $d=d_j \to 0$ such that
	\begin{equation}\label{yi}
		R\int_{\partial B_{R}(0)}u^{q_1+1}(z)ds +
		d\int_{\partial B_{d}(0)}u^{q_1+1}(z)ds \to 0.
	\end{equation}
	By the H\"older inequality, we obtain that for
	sufficiently large $R$,
	$$\begin{aligned}
	R \int_{\partial B_R(0)} \frac{u^{q_1}(z)ds}{|x-z|^{n-\gamma}}
	&\leq CR^{1-n+\gamma-\frac{q_1}{q_1+1}} \left(R\int_{\partial B_{R}(0)}u^{q_1+1}(z)ds\right)^{\frac{q_1}{q_1+1}}
	|\partial B_R(0)|^{\frac{1}{q_1+1}}\\
&\leq CR^{1-n+\gamma-\frac{q_1}{q_1+1}+\frac{n-1}{q_1+1}}
		\left(R\int_{\partial B_{R}(0)}u^{q_1+1}(z)ds\right)^{\frac{q_1}{q_1+1}}.
\end{aligned}
	$$
Letting $R=R_j \to \infty$ and using \eqref{yi}, we get
	\begin{equation}\label{bing}
		R \int_{\partial B_R(0)} \frac{u^{q_1}(z)ds}{|x-z|^{n-\gamma}}		
		\to 0.
	\end{equation}
Here we use
	\begin{equation}\label{Nec}
		q_1 > \frac{\gamma}{n-\gamma}
	\end{equation}
	which is implied by Theorem \ref{th10} (i).
	
	Similarly, when $d=d_j \to 0$,
	\begin{equation}\label{bing2}
		d \int_{\partial B_d(0)} \frac{u^{q_1}(z)ds}{|x-z|^{n-\gamma}}
		\leq Cd^{1-\frac{q_1}{q_1+1}} \left(d\int_{\partial B_{d}(0)}u^{q_1+1}(z)ds\right)^{\frac{q_1}{q_1+1}}
		|\partial B_d(0)|^{\frac{1}{q_1+1}} \to 0.
	\end{equation}
	
	Next, for $\Omega \subset \mathbb{R}^n$, write
	$$
	Q(\Omega):=\int_{\Omega} \frac{u^{q_1}(z)(x-z)\cdot z}
	{|x-z|^{n-\gamma+2}}dz.
	$$
	We claim that the improper integral $|Q(\mathbb{R}^n)|<\infty$
	for each $x \in \mathbb{R}^n \setminus \{0\}$.
	
	In fact, we observe that the defects of $Q(\mathbb{R}^n)$ may happen at $x$ and $\infty$.
	When $z$ is near $\infty$, by \eqref{Nec}, we can find some large $r>\max\{1,2|x|\}$ such that
	$$
	|Q(\mathbb{R}^n\setminus B_r(0))|
	\leq C\left(\int_{\mathbb{R}^n}u^{q_1+1}(z)dz\right)^{\frac{q_1}{q_1+1}}
	\left(\int_r^\infty \rho^{n-(q_1+1)(n-\gamma)} \frac{d\rho}{\rho}\right)^{\frac{1}{q_1+1}}
	<\infty.
	$$
	When $z$ is near $x$, we can find some small $\delta \in (0,1)$ such that
	$$
	|Q(B_\delta(x))|
	\leq C\left(\int_{\mathbb{R}^n}u^s(z)dz\right)^{\frac{q_1}{s}}
	\left(\int_0^\delta \rho^{n-\frac{s}{s-q_1}(n-\gamma+1)}
	\frac{d\rho}{\rho}\right)^{1-\frac{q_1}{s}}<\infty.
	$$
	Here we use the assumption of $u \in L^s(\mathbb{R}^n)$ for some $s>nq_1/(\gamma-1)$.
	Combining two results above, we prove $|Q(\mathbb{R}^n)|<\infty$.
	
	Now, we prove that (\ref{jia}). Integrating by parts yields
	\begin{equation}\label{geng}
		\begin{aligned}
			\int_{B_R(0) \setminus B_d(0)}&\frac{z\cdot\nabla u^{q_1}(z)dz}{|x-z|^{n-\gamma}}
			= R\int_{\partial B_R(0)} \frac{u^{q_1}(z)ds}{|x-z|^{n-\gamma}}
			-d\int_{\partial B_d(0)} \frac{u^{q_1}(z)ds}{|x-z|^{n-\gamma}}\\
			&  -n\int_{B_R(0) \setminus B_d(0)}\frac{u^{q_1}(z)dz}{|x-z|^{n-\gamma}}
			-(n-\gamma)\int_{B_R(0) \setminus B_d(0)} \frac{u^{q_1}(z)(x-z) \cdot z}{|x-z|^{n-\gamma+2}}dz.
		\end{aligned}
	\end{equation}
	Letting $R=R_j \to \infty$ and $d=d_j \to 0$ in (\ref{geng}),
	and using (\ref{bing}) and \eqref{bing2}, we can see that
	$$
	\lim_{d \to 0}\int_{\mathbb{R}^n \setminus B_d(0)}\frac{z\cdot\nabla u^{q_1}(z)dz}{|x-z|^{n-\gamma}}
	=-nv(x)|x|^\gamma+(\gamma-n)Q(\mathbb{R}^n),
	$$
	and hence \eqref{jia} holds at each $x \in \mathbb{R}^n \setminus \{0\}$.
	
	{\it Step 2.} For $x \neq 0$ and $\mu>0$, we have
	$$
	v(\mu x)=\int_{\mathbb{R}^n}\frac{u^{q_1}(y)dy}{|\mu
		x-y|^{n-\gamma}|\mu x|^{\gamma}}
	=\int_{\mathbb{R}^n}\frac{u^{q_1}(\mu
		z)dz}{|x-z|^{n-\gamma}|x|^{\gamma}}.
	$$
	Differentiating with respect to $\mu$ and then letting $\mu=1$, we get
	\begin{equation} \label{poho}
		x\cdot\nabla v(x)=\lim_{d \to 0}\int_{\mathbb{R}^n\setminus B_d(0)}\frac{z\cdot\nabla u^{q_1}(z)dz}
		{|x-z|^{n-\gamma}|x|^{\gamma}},
	\end{equation}
	which makes sense by virtue of \eqref{jia}.
	Multiplying \eqref{poho} by $v^{q_2}$ and integrating on $R^n \setminus B_d(0)$,
	we have
	\begin{equation} \label{shuix}
		\lim_{d \to 0} \int_{\mathbb{R}^n \setminus B_d(0)}v^{q_2}(x)(x\cdot\nabla v(x))dx
		=\lim_{d \to 0}\int_{\mathbb{R}^n\setminus B_d(0)}v^{q_2}(x)
		\int_{\mathbb{R}^n\setminus B_d(0)}\frac{z\cdot\nabla u^{q_1}(z)dz}
		{|x-z|^{n-\gamma}|x|^{\gamma}}dx.
	\end{equation}
	Integrating by parts, we see the left hand side of \eqref{shuix}
	$$\begin{aligned}
		K_1&:=\lim_{d \to 0} \int_{\mathbb{R}^n \setminus B_d(0)}v^{q_2}(x)(x\cdot\nabla v(x))dx
		=\lim_{d \to 0} \frac{1}{q_2+1} \int_{\mathbb{R}^n \setminus B_d(0)}(x\cdot\nabla v^{q_2+1}(x))dx\\
		&=\lim_{r \to \infty} \frac{r}{q_2+1}\int_{\partial B_r(0)}v^{q_2+1}(x)ds
		-\lim_{d \to 0} \frac{d}{q_2+1}\int_{\partial B_d(0)}v^{q_2+1}(x)ds
		-\frac{n}{q_2+1}\int_{\mathbb{R}^n}v^{q_2+1}(x)dx.
	\end{aligned}
	$$
	In view of $v \in L^{q_2+1}(\mathbb{R}^n)$, we can find $r=r_j \to \infty$
	and $d=d_j \to 0$ such that
	\begin{equation}\label{yi2}
		\lim_{r \to \infty} r\int_{\partial B_r(0)}v^{q_2+1}(x)ds
		=\lim_{d \to 0} d\int_{\partial B_d(0)}v^{q_2+1}(x)ds=0,
	\end{equation}
	and hence
	$$
	K_1=-\frac{n}{q_2+1}\int_{\mathbb{R}^n}v^{q_2+1}(x)dx.
	$$
	Using the Fubini theorem, we have
	$$\begin{aligned}
		K_2&:=\lim_{d \to 0}\int_{\mathbb{R}^n\setminus B_d(0)}v^{q_2}(x)
		\int_{\mathbb{R}^n\setminus B_d(0)}\frac{z\cdot\nabla u^{q_1}(z)dz}
		{|x-z|^{n-\gamma}|x|^{\gamma}}dx\\
		&=\lim_{d \to 0}\int_{\mathbb{R}^n\setminus B_d(0)}z\cdot\nabla u^{q_1}(z)
		\int_{\mathbb{R}^n\setminus B_d(0)}\frac{v^{q_2}(x)dx}
		{|x|^{\gamma}|z-x|^{n-\gamma}} dz\\
		&=\lim_{d \to 0}\int_{\mathbb{R}^n\setminus B_d(0)}(z\cdot\nabla u^{q_1}(z))u(z)dz.
	\end{aligned}
	$$
	Similar to the calculation of $K_1$, we also obtain
	$$
	K_2=-\frac{q_1n}{q_1+1}\int_{\mathbb{R}^n}u^{q_1+1}(z)dz.
	$$
	Inserting $K_1$ and $K_2$ into (\ref{shuix}), we have
	\begin{equation}\label{101}
-\frac{q_1n}{q_1+1}\int_{\mathbb{R}^n}u^{q_1+1}(x)dx=
	-\frac{n}{q_2+1}\int_{\mathbb{R}^n}v^{q_2+1}(z)dz.
\end{equation}

	By (\ref{j8}) and the Fubini theorem, we also have
	\begin{equation}\label{zxc}
		\begin{aligned}
			&\int_{\mathbb{R}^n}u^{q_1+1}(x)dx=\int_{\mathbb{R}^n}u^{q_1}(x)u(x)dx
			=\int_{\mathbb{R}^n}u^{q_1}(x)\int_{\mathbb{R}^{n}} \frac{v^{q_2}(y)}{|x-y|^{n-\gamma}|y|^{\gamma}}dx\\
			&=\int_{\mathbb{R}^n}v^{q_2}(y)\int_{\mathbb{R}^{n}} \frac{u^{q_1}(x)dx}{|y|^{\gamma}|y-x|^{n-\gamma}}dy
			=\int_{\mathbb{R}^n}v^{q_2}(y)v(y)dy=\int_{\mathbb{R}^n}v^{q_2+1}(y)dy.
		\end{aligned}
	\end{equation}
	Combining \eqref{101} and \eqref{zxc} we obtain (\ref{serrin}). This is $q_1q_2=1$.
	
	(ii) Suppose $(u,v) \in L^{q_1+1}(\mathbb{R}^n)
	\times [L^{q_2+1}(\mathbb{R}^n) \cap L^{s}(\mathbb{R}^n)]$
	for some $s>nq_2/(\gamma-1)$. We verify $q_1q_2=1$.
	
	{\it Step 1. We claim that the improper integral
	\begin{equation}\label{jia2}
		\lim_{d \to 0}\int_{\mathbb{R}^{n} \setminus B_d(0)}
		\frac{z\cdot\nabla v^{q_2}(z)}{|x-z|^{n-\gamma}|z|^\gamma}dz<\infty
	\end{equation}
	at each $x \in \mathbb{R}^n \setminus \{0\}$.}
	
	In fact,
	by the H\"older inequality and \eqref{yi2}, we obtain that
	\begin{equation}\label{bing21}
		R \int_{\partial B_R(0)} \frac{v^{q_2}(z)ds}{|x-z|^{n-\gamma}|z|^\gamma}
		\leq CR^{1-n-\frac{q_2}{q_2+1}}
		\left(R\int_{\partial B_{R}(0)}v^{q_2+1}(z)ds\right)^{\frac{q_2}{q_2+1}}
		|\partial B_R(0)|^{\frac{1}{q_2+1}} \to 0
	\end{equation}
	with $R=R_j \to \infty$. Noting
	\begin{equation}\label{Nec2}
		0<q_2 < \frac{n-\gamma}{\gamma}
	\end{equation}
	which is implied by Theorem \ref{th10} (i), and using (\ref{yi2}) we get
	\begin{equation}\label{bing22}
		d \int_{\partial B_{d}(0)} \frac{v^{q_2}(z)ds}{|x-z|^{n-\gamma}|z|^\gamma}
		\leq \frac{Cd^{1-\gamma}}{|x|^{n-\gamma}} d^{-\frac{q_2}{q_2+1}}
		\left(d\int_{\partial B_{d}(0)}v^{q_2+1}(z)ds\right)^{\frac{q_2}{q_2+1}}
		|\partial B_d(0)|^{\frac{1}{q_2+1}}
		\to 0
	\end{equation}
	with $d=d_j \to 0$.
	
	Next, for $\Omega \subset \mathbb{R}^n$, write
	$$
	L(\Omega):=\int_{\Omega} \frac{v^{q_2}(z)(x-z)\cdot z}
	{|x-z|^{n-\gamma+2}|z|^\gamma}dz.
	$$
	We first prove $|L(\mathbb{R}^n)|<\infty$
	for each $x \in \mathbb{R}^n \setminus \{0\}$.
	
	In fact, we observe that the defects of $L(\mathbb{R}^n)$ may happen at $x$, $0$ and $\infty$.
	When $z$ is near $\infty$, for some large $r>\max\{1,2|x|\}$, we have
	$$
	|L(\mathbb{R}^n\setminus B_r(0))|
	\leq C\left(\int_{\mathbb{R}^n}v^{q_2+1}(z)dz\right)^{\frac{q_2}{q_2+1}}
	\left(\int_r^\infty \rho^{n-n(q_2+1)} \frac{d\rho}{\rho}\right)^{\frac{1}{q_2+1}}
	<\infty.
	$$
	When $z$ is near $0$, for some small $\delta \in (0,\min\{1,|x|/2\})$, by \eqref{Nec2} we have
	$$
	|L(B_\delta(0))|
	\leq C\left(\int_{\mathbb{R}^n}v^{q_2+1}(z)dz\right)^{\frac{q_2}{q_2+1}}
	\left(\int_0^\delta \rho^{n-(\gamma-1)(q_2+1)} \frac{d\rho}{\rho}\right)^{\frac{1}{q_2+1}}
	<\infty.
	$$
	When $z$ is near $x$, we can find some small $\delta \in (0,1)$ such that
	$$
	|L(B_\delta(x))|
	\leq C\left(\int_{\mathbb{R}^n}v^s(z)dz\right)^{\frac{q_2}{s}}
	\left(\int_0^\delta \rho^{n-\frac{s}{s-q_2}(n-\gamma+1)}
	\frac{d\rho}{\rho}\right)^{1-\frac{q_2}{s}}<\infty
	$$
	by virtue of $v \in L^s(\mathbb{R}^n)$ for some $s>nq_2/(\gamma-1)$.
	Combining two results above, we prove $|L(\mathbb{R}^n)|<\infty$.
	
	Finally we prove that (\ref{jia2}). Integrating by parts yields
	$$\begin{aligned}
		\int_{B_R(0) \setminus B_d(0)}
		&\frac{z\cdot\nabla v^{q_2}(z)dz}{|x-z|^{n-\gamma}|z|^\gamma}
		= R\int_{\partial B_R(0)} \frac{v^{q_2}(z)ds}{|x-z|^{n-\gamma}|z|^\gamma}
		-d\int_{\partial B_d(0)} \frac{v^{q_2}(z)ds}{|x-z|^{n-\gamma}|z|^\gamma}\\
		&  -(n-\gamma)\int_{B_R(0) \setminus B_d(0)}\frac{v^{q_2}(z)dz}{|x-z|^{n-\gamma}|z|^\gamma}
		-(n-\gamma)\int_{B_R(0) \setminus B_d(0)} \frac{v^{q_2}(z)(x-z) \cdot z}{|x-z|^{n-\gamma+2}|z|^\gamma}dz.
	\end{aligned}
	$$
	Letting $R=R_j \to \infty$ and $d=d_j \to 0$,
	and using (\ref{bing21}) and \eqref{bing22},  we obtain
	$$
	\lim_{d \to 0}\int_{\mathbb{R}^n \setminus B_d(0)}\frac{z\cdot\nabla v^{q_2}(z)dz}{|x-z|^{n-\gamma}|z|^\gamma}
	=(\gamma-n)u(x)+(\gamma-n)L(\mathbb{R}^n),
	$$
	and hence \eqref{jia2} holds at each $x \in \mathbb{R}^n \setminus \{0\}$.
	
	{\it Step 2.} Similar to the derivation of \eqref{poho}, for $x \neq 0$ we get
	$$
	x\cdot\nabla u(x)=\lim_{d \to 0}\int_{\mathbb{R}^n\setminus B_d(0)}\frac{z\cdot\nabla v^{q_2}(z)dz}
	{|x-z|^{n-\gamma}|z|^{\gamma}}.
	$$
	Multiplying this result by $u^{q_1}$ and integrating on $R^n \setminus B_d(0)$,
	we have
	\begin{equation} \label{shuix2}
		\lim_{d \to 0} \int_{\mathbb{R}^n \setminus B_d(0)}u^{q_1}(x)(x\cdot\nabla u(x))dx
		=\lim_{d \to 0}\int_{\mathbb{R}^n\setminus B_d(0)}u^{q_1}(x)
		\int_{\mathbb{R}^n\setminus B_d(0)}\frac{z\cdot\nabla v^{q_2}(z)dz}
		{|x-z|^{n-\gamma}|z|^{\gamma}}dx.
	\end{equation}
	Calculating the same as in $K_1$ and $K_2$, and using \eqref{yi} and \eqref{yi2} we get
	$$
	K_3:=\displaystyle\lim_{d \to 0} \int_{\mathbb{R}^n \setminus B_d(0)}u^{q_1}(x)(x\cdot\nabla u(x))dx
	=-\displaystyle\frac{n}{q_1+1}\int_{\mathbb{R}^n}u^{q_1+1}(x)dx,
	$$
	$$
	K_4:=\displaystyle\lim_{d \to 0}\int_{\mathbb{R}^n\setminus B_d(0)}u^{q_1}(x)
	\int_{\mathbb{R}^n\setminus B_d(0)}\frac{z\cdot\nabla v^{q_2}(z)dz}
	{|x-z|^{n-\gamma}|z|^{\gamma}}dx
	=-\frac{q_2n}{q_2+1}\int_{\mathbb{R}^n}v^{q_2+1}(z)dz.
	$$
	Combining two results above with (\ref{shuix2}), we have
	$$
	-\frac{n}{q_1+1}\int_{\mathbb{R}^n}u^{q_1+1}(x)dx=
	-\frac{n q_2}{q_2+1}\int_{\mathbb{R}^n}v^{q_2+1}(z)dz.
	$$
	This result, together with \eqref{zxc}, implies (\ref{serrin}). This is $q_1q_2=1$.

\paragraph{Acknowledgements.}
This work was supported by the Natural Science Foundation of Jiangsu (No. BK20241878).

	{\sc Tiantian Zhou}
	
	Institute of Mathematics,
	School of Mathematical Sciences
	
	Nanjing Normal University,
	Nanjing, 210023, China
	
	Email:zhoutiantiannj@163.com
	
	\vskip 5mm
	
	{\sc Yutian Lei}
	
	Ministry of Education Key Laboratory for NSLSCS, School of Mathematical Sciences
	
	Nanjing Normal University, Nanjing, 210023, China
	
	Email: leiyutian@njnu.edu.cn

\end{document}